\documentclass[10pt,twoside]{amsart}
\usepackage[all]{xy}
\usepackage{amssymb,latexsym,amsthm,amsmath,mathtools,tikz}
\usepackage{enumitem,color}

\usepackage{mathtools}
\numberwithin{equation}{section}
\newtheorem{thm}{Theorem}[section]
\newtheorem{cor}[thm]{Corollary}
\newtheorem{lemma}[thm]{Lemma}

\newtheorem{prop}[thm]{Proposition}

\newtheorem*{thma'}{Theorem A'}

\theoremstyle{definition}
\newtheorem{defn}[thm]{Definition}
\newtheorem{example}[thm]{Example}
\newtheorem{remark}[thm]{Remark}
\usepackage{multirow}
\usepackage{float}
\usepackage{longtable}
\usepackage{tikz}
\usepackage{placeins}
\usepackage{array}
\newcolumntype{L}{>{\centering\arraybackslash}m{7cm}}
\usepackage{tabularray}
\tikzstyle{vertex}=[circle,fill=black, draw, inner sep=0pt, minimum size=7pt]
\newcommand{\vertex}{\node[vertex]}
\usetikzlibrary{decorations.markings}
\usetikzlibrary{decorations.pathreplacing}
\usetikzlibrary{arrows.meta}
\usepackage{anyfontsize}
\usepackage{graphicx}

\newcommand{\dlabel}[1]{\ifmmode \text{\ttfamily \upshape [#1] } \else
{\ttfamily \upshape [#1] }\fi \label{#1} }

\newcommand{\gen}[1]{\left < #1 \right >}

\newcommand{\diag}{\operatorname{diag} }

\begin{document}
\setcounter{section}{0}

\title{Waring problem for triangular matrix algebra}
\author{Rahul Kaushik}
\address{Indian Institute of Science Education and Research Pune, 411008, Maharashtra, India}
\email{rahul.kaushik@acads.iiserpune.ac.in}
%\thanks{}
\author{Anupam Singh}
\address{Indian Institute of Science Education and Research Pune, 411008, Maharashtra, India}
\email{anupamk18@gmail.com}
\thanks{The first named author is funded by NBHM postdoctoral fellowship 0204/4/2023/R\&D-II/884 for this work. The second named author is funded by NBHM through a research project 02011/23/2023/NBHM(RP)/ RDII/5955 for this work.}

%%%%%%%%%%%%%%%%%%%%%
\begin{abstract} 
The Matrix Waring problem is if we can write every matrix as a sum of $k$-th powers. Here, we look at the same problem for triangular matrix algebra $T_n(\mathbb{F}_q)$ consisting of upper triangular matrices over a finite field. We prove that for all integers $k, n \geq 1$, there exists a constant $\mathcal C(k, n)$, such that for all $q> \mathcal C(k,n)$, every matrix in $T_n(\mathbb{F}_q)$ is a sum of three $k$-th powers. Moreover, if $-1$ is $k$-th power in $\mathbb{F}_q$, then  for all $q>\mathcal C(k,n)$, every matrix in $T_n(\mathbb{F}_q)$ is a sum of two $k$-th powers. We make use of Lang-Weil estimates about the number of solutions of equations over finite fields to achieve the desired results.
\end{abstract}

\maketitle

%%%%%%%%%%%%%%%%%%%%%%%
\section{Introduction}

The Matrix Waring problem is a generalisation of the classical Waring problem in Number Theory to a non-commutative setup which is as follows. Let $R$ be a ring with identity and $M_n(R)$ be the $n$-by-$n$ matrices over $R$. For a given $k$, a positive integer, the Matrix Waring problem asks if there exists a (smallest if it exists) number $g(k)$ such that every element of $M_n(R)$ can be written as a sum of $g(k)$ many $k$-th powers. That is, for all $A\in M_n(R)$ the equation $X_1^k + X_2^k + \cdots + X_{g(k)}^k =A$ has a solution. Many interesting results are known in this direction, mainly due to the work of Vaserstein, Richman, Katre, Garge, Bresar and Semrl (see~\cite{Va87, Ri87, Ga21, KG13, BS22, BS23}) to name a few. This problem can be also thought of as a particular case of a more general problem about studying images of polynomial maps on rings. 

In recent years the question of surjectivity, arising from word maps on groups (defined by evaluation using an element of a free group called words in $m$-variable) has been a central problem in group theory. Several deep results have been proved for finite simple groups, especially by Larsen, Shalev and Tiep (see for example~\cite{LST11, LST19}). Parallel to these problems, polynomial maps on algebras (defined by evaluation using an element of a polynomial ring in $m$-non-commuting variables) have been well studied (see the survey article by Kanel-Belov, Malev, Rowen and Yavich~\cite{BMRY20} for more information). In recent work, Kishore \cite{KK22}, Kishore-Singh \cite{KA22}, motivated by a question of Larsen, proved the following: for large enough $q$, every element of $M_n(\mathbb F_q)$ is a sum of two $k$-th powers. In the present work, we begin an exploration of this problem for algebras beyond central simple algebra, for example, if we can get results as suggested by Larsen for upper triangular matrix algebra $T_n(\mathbb F_q)$. An analogous problem in the case of triangular groups namely the Verbal width for triangular matrix group and $p$-powers is studied in~\cite{So16, DSY21} respectively. In a recent work, Panja and Prasad~\cite{SP23} have characterised images of polynomials on $T_n(\mathbb K)$, when $\mathbb K$ is algebraically closed.  

Let $T_n(\mathbb F_q)$ (respectively $B_n(\mathbb F_q)$, $U_n(\mathbb F_q)$, $N_n(\mathbb F_q)$) be the set of $n \times n$  upper triangular matrices (respectively nonsingular,  unipotent, nilpotent upper triangular matrices). Now our main question is as follows: 
\begin{quote}
When can we write every element of $T_n(\mathbb{F}_q)$ as a sum of two (or three) $k$-th powers over large enough finite fields? 
\end{quote}
To answer the above question we prove the following:

\begin{thm}\label{thma}
Let $k\geq 2, n \geq 2$ be integers. Then,
\begin{enumerate}
\item there exists a constant $\mathcal C(k,n)$ depending on $k $ and $n$ such that for all $q> \mathcal C(k, n)$, every element of $T_n(\mathbb{F}_q)$ can be written as a sum of three $k$-th   powers $X^k + Y^k + Z^k$, where $X$ is diagonalizable and $Y, Z$ are diagonal matrices. The constant $\mathcal C(k, n)$ can be taken to be $4n^2k^{16}$.
\item Further, if $-1$ is a $k$-th power in $\mathbb{F}_q$, then for all $q> \mathcal C({k,n})=4n^2k^{16}$, every element of $T_n(\mathbb{F}_q)$  can be written as sum of two $k$-th   powers $X^k + Y^k$, where $X$ is diagonalizable and $Y$ is a diagonal matrix. 
\end{enumerate}
\end{thm}
\noindent Note that writing an element of $T_n(\mathbb{F}_q)$ as a sum of two $k$-th powers in $M_n(\mathbb{F}_q)$ is not helpful. In Example~\ref{example-necessary2} we give an example of the same. In Example~\ref{example-necessary} we show that the condition $-1$ being a $k$-th root is a necessary condition for writing an upper triangular matrix as a sum of two $k$-th powers.

Our key idea is to use Lang-Weil estimates on the number of solutions of an equation over finite fields. A crude way to look at our problem is to simultaneously solve $n(n+1)/2$ equations obtained from $A=X_1^k + X_2^k$ over a finite field and see if a solution exists. Thus, the idea that over a large enough finite field $\mathbb F_q$ we could obtain such a solution is what we need. This is done in~\cite{KK22, KA22} for $M_n(\mathbb F_q)$.  

To write an element as a sum of $k$-th powers, it is enough to do this for a similarity class (sometimes we may also call it conjugacy class) representative. That is, if $A = X_1^k + X_2^k +\cdots $, then $P^{-1}AP = (P^{-1}X_1P)^k + (P^{-1}X_2P)^k +\cdots$, where $P$ is an invertible matrix in $T_n(\mathbb{F}_q)$. However, the similarity in $T_n(\mathbb F_q)$ is trickier than that of in $M_n(\mathbb{F}_q)$. For any $A,B \in T_n(\mathbb{F}_q)$, we say that $A$ and $B$ are $B_n$-similar or conjugate, if there exist some $P \in B_n(\mathbb F_q)$ such that $B=P^{-1}AP$. For each $B_n$-similarity class, there is a nice representative which is called Belitski\v{i}'s canonical form of the similarity or conjugacy class. For every matrix $A$,  Belitski\v{i}'s canonical form of $A$ will be denoted by $A^{\infty}$. We recall this in Section~\ref{section-similarity-classes}. Given this, our problem gets reduced to writing every representative of similarity class of $T_n(\mathbb{F}_q)$ as a sum of two or three $k$-th powers. We prove our main theorem in Section~\ref{section-proof-theorem}. An exact description of the $B_n$-similarity/conjugacy classes is described in~\cite{MMH} up to $8 \times 8$ upper triangular matrices. We use the same to give an explicit computation in Section~\ref{section-table} of how we can write each representative as a sum of two $k$-th powers. 

We denote by $p$ the characteristic of the field $\mathbb F_q$.  By $\mathbb F_q^m$ we mean the direct product $\mathbb F_q \times \ldots \times \mathbb F_q$ $m$-times. We denote by $E_{rs}$, the elementary matrix with the $(r,s)$ entry $1$ and $0$ elsewhere.

%%%%%%%%%%%%%%%%%%%%%%%%%%
\section{Some examples and computations in $T_n(\mathbb F_q)$}

Let us set 
$$f(x,y) = x^{k-1} + x^{k-2}y+ \cdots + xy^{k-2} + y^{k-1}$$ to be a homogeneous polynomial of degree $k-1$ in $\mathbb{F}_q[x,y]$. 
We begin with the following lemma, which tells about the zeroes of the polynomial $f(x,y)$.
\begin{lemma}\label{homeqn}
Let $k > 1$ be any integer. Consider $f(x,y) = x^{k-1} + x^{k-2}y+ \cdots + xy^{k-2} + y^{k-1}$, a homogeneous polynomial of degree $k-1$ in $\mathbb{F}_q[x,y]$. Then, there exists $a, b \in \mathbb{F}_q$, not both zero, with $f(a,b)=0$ if and only if one of the following happens (i) $a=b$ and $p\mid k$, (ii) $a\neq b$ and $a^k=b^k$. 
\end{lemma}
\begin{proof}
Let $h(x)= x^{k-1}+ x^{k-2}+ \cdots + x+1 $ be a polynomial in $\mathbb{F}_q[x]$. Then \[ h(x) = \begin{cases}   \frac{x^k-1}{x-1} & x \neq 1,\\ k & x=1. \end{cases}   \]
Now it is clear that $h(\alpha)=0$ if and only if one of the following happens: (i) $\alpha = 1 $ and $p\mid k$, (ii) $\alpha \neq 1$ and $\alpha^k = 1$. We can express $f$ as follows, 
\[ f(x,y)  = \begin{cases} y^{k-1}h(x/y) & y \neq 0,\\ x^{k-1} & y=0. 
\end{cases}\]
Thus, $f(x,y) = 0$ if and only if $x = y = 0$ or $h(x/y)=0$, which implies either $x=y$ and $p\mid k$ or $x\neq y$ and $x^k=y^k$.  Hence there exist $a, b \in \mathbb{F}_q$, not both zero, with $f(a,b)=0$ if and only if either $a=b$ and $p\mid k$, or $a\neq b$ and $a^k=b^k$.
\end{proof}

\noindent In the following lemma, first, we see that, not every matrix of $T_n(\mathbb{F}_q)$ can be written as a $k$-th   power (compare this with Lemma~\ref{mult}). 
\begin{lemma}\label{2*2}
Suppose $p \mid k$. Let $C$ be a matrix of the form \[ \begin{bmatrix} \beta^k & \alpha   \\     0 & \beta^k \\ \end{bmatrix}, \] 
where $k \geq 2$ and $\alpha \neq 0, \beta \in \mathbb{F}_q$. Then $C\neq A^k$ for any $k \geq 2$ and for any $A \in T_n(\mathbb{F}_q)$.
\end{lemma}
\begin{proof}
Let $A = \begin{bmatrix}   a & b   \\    0 & c \\ \end{bmatrix}$ be in $T_n(\mathbb{F}_q)$ with $A^k=C$. Then, $A^k = \begin{bmatrix}   a^k & bf(a,c)   \\    0 & c^k \\ \end{bmatrix}$ gives $a^k=c^k=\beta^k$ and $bf( a,c)=\alpha$. If $\beta=0$, then $a=0=c$.  Then $f(a,c)=0$ implies $\alpha=0$, which is not possible. Now let us assume that $\beta \neq 0$.  Then  $a$ and $c$ can not be $0$. If $a=c$,  then $f(a,c)=ka^{k-1}$ and hence we get $bka^{k-1} =\alpha$, which gives us $\alpha=0$ as $p\mid k$, a contradiction. Thus we must have $a\neq c$. Since  $a^k=c^k$, and therefore by Lemma~\ref{homeqn} we get that $f(a,c)=0$, which gives $\alpha=0$, not possible. Thus $C\neq A^k$ for any $k \geq 2$ and $A \in T_n(\mathbb{F}_q)$.
\end{proof}

\begin{example}\label{example-necessary}
In Theorem~\ref{thma}(2), the condition of $-1$ being $k$-th  power is necessary to write every matrix as a sum of two $k$-th   powers. To see this, let us assume that the following equation of matrices holds in $T_2(\mathbb{F}_q)$,
\[\begin{bmatrix}    0 & 1   \\     0 & 0  \\ \end{bmatrix} = \begin{bmatrix}     a & b   \\     0 & c  \\ \end{bmatrix}^k + \begin{bmatrix}     u & v   \\     0 & w  \\ \end{bmatrix}^k. 
\]
This leads to the following system of equations in $\mathbb{F}_q$, $ a^k+ u^k = 0, c^k + w^k = 0$ and $b(a^{k-1}+ca^{k-2}+\cdots +c^{k-1})+ v(u^{k-1}+wu^{k-2}+\cdots +w^{k-1})  = 1$. Note that $a=0$ (similarly $c=0$) if and only if $u=0$ ($w=0$). If $a=0$ and $c=0$, then the last equation gives a contradiction.  Thus, either $a\neq 0$ or $c\neq 0$. In that case, we must have either $(au^{-1})^k=-1$, or $(cw^{-1})^k=-1$. This implies that in either  of the case $-1$ is a $k$-th  power in $\mathbb{F}_q$. Thus, the condition of $-1$ being $k$-th power can not be dropped from Theorem~\ref{thma}(2). 

\noindent By a similar computation, we can show that if $-1$ is not $k$-th  power in $\mathbb{F}_q$, then the Jordan nilpotent matrices of the form 
$$\begin{bmatrix} 0&1&0 &\cdots &0 \\ &0&1& \cdots &0 & \\ &&\ddots&\ddots&\vdots \\&&&0&1 \\ &&&&0\end{bmatrix}$$
is also not a sum of two $k$-th   powers in $T_n(\mathbb F_q)$.

%\noindent A particular situation of the above is as follows. If the characteristic of $\mathbb{F}_q$ is even, then $-1$ is not a $k$-th   power. Since in that case, the order of cyclic group $\mathbb{F}_q^*$ is $2^m-1$, which is odd. Now if there exists an element $x \in \mathbb{F}_q$ such that $x^k=-1$, then the order of $x$ will be $2k$, which must divide $2^m-1$, which is not possible. Thus the above nilpotent matrices in $T_n(\mathbb F_q)$ can not be written as a sum of two $k$-th   powers for any $k\geq 2$, when the characteristic of $\mathbb F_q$ is even.
\end{example}

\begin{example}\label{example-necessary2}
Suppose $-1$ is not a $k$-th power in $\mathbb F_q$. Then, from the  above mentioned example $\begin{bmatrix}    0 & 1   \\     0 & 0 \end{bmatrix} $ is  not a sum of two $k$-th powers in $T_2(\mathbb{F}_q)$. However, from ~\cite[Proposition 3.5]{KK22} for large enough $q$, $\begin{bmatrix}    0 & 1   \\     0 & 0  \\ \end{bmatrix} $ is a sum of two $k$-th powers in $M_2(\mathbb{F}_q)$. 
\end{example}

Now, we discuss when we get a $0$ entry after multiplying two upper triangular matrices as it will be required later.
\begin{lemma}
Let $A, B \in T_n(\mathbb{F}_q)$ such that $a_{rs} b_{st} = 0$ for every $r< s< t$. If for some $1 \leq i,j \leq n$,  $a_{ij} = b_{ij} = 0$, then $(AB)_{ij}=0$.
\end{lemma}
\begin{proof}
It is easy to check that 
\begin{eqnarray*}
(AB)_{ij} & = & \sum_{l=i}^{j} a_{il}b_{lj} = a_{ii}b_{ij}+ a_{i,i+1}b_{i+1,j}+ \ldots +  a_{i,j-1}b_{j-1,j} + a_{ij}b_{jj} \\
 & = & a_{i,i+1}b_{i+1,j} \ldots +  a_{i,j-1}b_{j-1,j} = 0.
\end{eqnarray*} 
\end{proof}

\begin{lemma}\label{zero}
Let $A \in T_n(\mathbb{F}_q)$ such that $a_{rs}a_{st}=0$ for every $r <s <t$. Then, for any positive integer $m \geq 1$, $(A^m)_{rs}a_{st} = 0$, for all $r <s <t$. Moreover, if for any $i \neq j$,  $a_{ij}=0$, then $({A^m})_{ij}=0$.
\end{lemma}
\begin{proof}
We will prove this by induction. By the previous lemma, we get that $(A^2)_{ij}= 0$.  Let $(A^{m-1})_{rs}a_{st}=0$, for all $r <s <t$.   Then, 
\begin{eqnarray*}
(A^{m})_{rs}a_{st} & =& \big( \sum_{l=r}^{s}(A^{m-1})_{rl}a_{ls} \big) a_{st}\\ 
& =& \big( (A^{m-1})_{rr}a_{rs}+ (A^{m-1})_{r,r+1}a_{r+1,s} + \cdots + (A^{m-1})_{rs}a_{ss} \big)a_{st}\\
& =& \big( (A^{m-1})_{rr}a_{rs} + (A^{m-1})_{rs}a_{ss} \big)a_{st} \\
& =&  (A^{m-1})_{rr}a_{rs}a_{st} + (A^{m-1})_{rs}a_{ss} a_{st} =0.
\end{eqnarray*}
Thus, the first statement is true for all $m \geq 1$. As $a_{ij}=0$.  Now let us assume that $(A^{l})_{ij}=0$ for all $l\leq m-1$. Since $(A^i)_{rs}a_{st}=0$ for all $i>0$,  $r<s<t$, therefore by using induction hypothesis we get,
\begin{eqnarray*}
(A^{m})_{ij} & =& \sum_{r=i}^{j}\big( (A^{m-1})_{ir} a_{rj}\big)  =  (A^{m-1})_{ii}a_{ij}+ (A^{m-1})_{i,i+1}a_{i+1,j} \ldots +  (A^{m-1})_{ij}a_{jj}\\
& =& (A^{m-1})_{ii}a_{ij} + (A^{m-1})_{i,i+1}a_{i+1,j} \ldots +  (A^{m-1})_{ij-1}a_{j-1,j}\\
& = & (A^{m-1})_{ii}a_{ij} = 0.
\end{eqnarray*}
Thus the result holds for every positive integer $m \geq 1$.    
\end{proof}

%%%%%%%%%%%%%%%%%%%%
\section{Equations over finite fields}\label{section-equations}

Our main idea is to use the existence of a reasonably large number of solutions to certain equations over the base field $\mathbb F_q$ to conclude about the solutions of the equation $X_1^k + \cdots + X_m^k = A$ over $M_n(\mathbb F_q)$ for an $A\in M_n(\mathbb F_q)$. In what follows we will use various estimates for the number of solutions to equations over finite fields, which is due to Lang and Weil (see~\cite{LW54}), however, we will primarily refer to the monograph by Schmidt~\cite{sch} for our work here. The following theorem is a slightly modified version of the Lang-Weil estimate, \cite[Theorem A.1]{KA22}.

\begin{thm}\label{lang}
Consider the following polynomial in $\mathbb{F}_q[X_1,\ldots, X_m]$:
$$F(X_1,\ldots,X_m):=\alpha_1 X_1^k + \alpha_2 X_2^k +\cdots + \alpha_m X_m^k -1,$$
where $\alpha_i  \neq 0$, for all $1\leq i \leq m$. Let $\mathcal N$ be the number of zeroes of the polynomial $F(X_1,\ldots, X_m)$ in $\mathbb{F}_q^m$. Then, we have
 $$|\mathcal N - q^{m-1}| \leq k^{2m}\sqrt{ q^{m-1}}. $$
\end{thm}

\noindent The following proposition A.2~\cite{KA22} tells us about the existence of at least two solutions of equation $X_1^k + Y_1^k= \lambda$ over $\mathbb{F}_q$, which is proved using Theorem \ref{lang}.

\begin{prop}\label{2var}
Let $k\geq 2$ and  the characteristic of $\mathbb{F}_q$ be not $2$. Let $\lambda$ be a non zero element in $\mathbb{F}_q$. Consider the equation 
\begin{equation}%\label{2.1}
    X^k + Y^k=\lambda.
\end{equation} 
Then, for all $q > k^{16}$, the equation~\eqref{2.1} has at least $2$ solutions $(x_1, y_1)$ and $ (x_2, y_2)$ in $\mathbb{F}_q^2$ such that $x_1^k \neq x_2^k$ and $y_1^k \neq y_2^k$. 
\end{prop}

\noindent In the next Lemma we will show the existence of at least $n$ solutions of the above kind. 

\begin{prop}\label{2var1}
Let $k\geq 2$ and $n\geq 2$ be positive integers and  $\lambda$ be a non zero element in $\mathbb{F}_q$. Consider the equation 
\begin{equation}\label{2.1}
    X^k + Y^k = \lambda.
\end{equation}  
Then, for all $q > n^2k^{16}$, the equation~\eqref{2.1} has at least $n$ solutions $(x_i, y_i) \in \mathbb{F}_q^2$, $1\leq i \leq n$, satisfying the following conditions:
\begin{enumerate}
\item  $x_i^k \neq y_i^k$ for all $ 1 \leq i \leq n$.
\item $x_i^k \neq x_j^k$ and $y_i^k \neq y_j^k$ for every $i \neq j$.
\end{enumerate}
\end{prop}
\begin{proof}
If we take $\alpha_1= \alpha_2 = 1/ \lambda$ in Theorem~\ref{lang}, then the number of solutions of equation \eqref{2.1}, say  $\mathcal N$,  satisfies the following inequality $|\mathcal N - q|\leq k^4 \sqrt{q}$. Let $S$ be the set of solutions of equation \eqref{2.1}, in $\mathbb F_q^2$.  For every $(u,v) \in S$, define $V_{u,v}=\{(x, y) \in S \mid x^k = u^k, \ y^k = v^k  \}$ be subset of $S$ in $\mathbb F_q^2$. By Theorem \ref{lang}, if $q>k^{16}$, then $0 < q-k^4\sqrt{q}<\mathcal{N}$, implies $S\neq \phi$. Hence at least one such $V_{u,v}$ exists. In fact, in the later part of the proof we will show that, for all $q > n^2k^{16}$ there exist at least $n$ many disjoint $V_{u,v}$'s, say $V_1, \ldots, V_r$, where $r \geq n$. These  $V_{u,v}$ will satisfy the following conditions:
\begin{enumerate}
\item For every $(u,v), (a,b) \in S$, either $V_{u,v} \cap V_{a,b} = \phi$ or $V_{u,v} = V_{a,b}$ and  $|V_{u,v}| \leq k^2$ for all $(u,v) \in S$.
\item  If $V_1,\ldots,V_r$ are all disjoint subsets of $S$, where $V_i=V_{u,v}$ for some $(u,v) \in S$, then $S= \mathop{\dot{\bigcup}} V_{i}$.
\end{enumerate}
  
(1). Let the characteristic of $\mathbb F_q$ be not $2$. Let $(u,v) \in S$ be any element.  If $u^k = v^k$, then  $u^k+v^k=2u^k=\lambda$. Since $x^k=\lambda/2$ can have at most $k$ solutions in $\mathbb F_q^*$, therefore $|V_{u,v}|\leq k^2$, where $u^k=v^k$. Now let $(u,v ) \in S$ be an element such that $u^k \neq v^k$. For any $(x,y) \in V_{u,v}$,  we get $x^k=u^k$ and $y^k =v^k$, which implies $x=\zeta_k u$ and $y=\zeta_k' v$ for some $k$-th root of unity $\zeta_k$ and $\zeta_k'$ in $\mathbb F_q$ (if exist). Thus every element of $V_{u,v}$ will be of the form $(\zeta_k u, \zeta_k' v)$. Notice that, for every element $(x,y) \in V_{u,v}$, we have $(y,x) \notin V_{u,v}$. Otherwise, if $(y,x) \in V_{u,v}$, then by definition of $V_{u,v}$, we get that $x^k=u^k=y^k=v^k$, which is not possible. Thus every element of $V_{u,v}$ will be of the type $(\zeta_k u, \zeta_k' v)$, and hence $|V_{u,v}| \leq k^2$. Now, we will prove that either $V_{u,v} \cap V_{a,b} = \phi$ or $V_{u,v} =V_{a,b}$. Let $(x,y) \in V_{u,v} \cap V_{a,b}$, where $u^k \neq v^k $ and $a^k \neq b^k$. Then $u^k=x^k=a^k$ and $v^k=y^k=b^k$.   Now for every $ (u_1,v_1) \in V_{u,v}$ and $(a_1,b_1) \in V_{a,b}$, and by definition of $V_{u,v}$ and $V_{a,b}$,  we get  $u_1^k=u^k=x^k=a^k=a_1^k$ and $v_1^k=v^k=y^k=b^k=b_1^k$,  which implies  $(a_1,b_1 ) \in V_{u,v}$ and $(u_1,v_1) \in V_{a,b}$. Thus we get  $V_{u,v}=V_{a,b}$. 
Now let $(x,y) \in V_{u,v} \cap V_{a,b}$, where $u^k = v^k $. Then $a^k=x^k=u^k=v^k=y^k=b^k$. for every $ (u_1,v_1) \in V_{u,v}$ and $(a_1,b_1) \in V_{a,b}$. Thus by definition of $V_{u,v}$ and $V_{a,b}$,  we get  $u_1^k=x^k=a_1^k$ and $v_1^k=y^k=b_1^k$,  which implies  $(a_1,b_1 ) \in V_{u,v}$ and $(u_1,v_1) \in V_{a,b}$. Thus we get  $V_{u,v}=V_{a,b}$. It also clear that, if $u^k=v^k$ and $a^k=b^k$, then $V_{u,v}=V_{a,b}$.  Now let the characteristic of $\mathbb F_q$ be $2$. Let us assume that $u^k=v^k$ for some  $(u,v) \in S$. Then we have $u^k+v^k=2u^k=0=\lambda$, which is not possible. Hence $u^k \neq v^k $ for every $(u,v) \in S$. The proof, in this case, goes on the same lines, when the  characteristic of $\mathbb F_q$ is  not $2$. 

(2). As all $V_{u,v}$ are either disjoint or same. We call those disjoints sets $V_1, \ldots, V_r$. If $u^k=v^k$ for some $(u,v) \in S$, then we set $V_1=V_{u,v}$. Let $(a,b ) \in S$ be any element. If $a^k=b^k$, then $(a,b) \in V_1$, otherwise $(a,b) \in V_i$ for some $ 2 \leq i \leq r$. Thus $S= \mathop{\dot{\bigcup}} V_{i} $.

As all $V_i$ are disjoint and $|V_i|\leq k^2$, therefore if $|S| > rk^2$, then  there exist at least $r$ many disjoint $V_i$. Thus, $|S| \leq k^2(r+1)$, where $r$ is the number of distinct $V_i$. As we know that $V_i \cap V_j = \phi$, therefore we can take one tuple from each $V_i$, which gives us as many required solutions as many $V_i$ are there. Thus, to have at least $n$ required solutions, we need to show that $\mathcal N \geq q - k^4 \sqrt{q} > k^2n$, where $|S|= \mathcal N$. As $n \geq 2$ and $k \geq 2$, therefore if we take $\mathcal C(k,n) = n^2k^{16}$, then for all $q > \mathcal C(k,n)$,  we get 
$$q - k^4\sqrt{q} > n^2k^{16}- k^4 \sqrt{n^2k^{16}} = k^{16} n(n-1) > k^2 n.$$
Thus, for all $q > \mathcal{C}(k,n) = n^2k^{16}$, we get at least $n$ required solutions, as we needed.
\end{proof}

\begin{lemma}\label{zero1}
Let $k\geq 1$, $m \geq 1$ and $q=p^m$ such that $-1$ is $k$-th power in $\mathbb F_q$. Then the equation $X_1^k + X_2^k = 0$ has exactly $\frac{q-1}{(k, \ q-1)} + 1$ number of solutions in $\mathbb{F}_q^2$ with the property if $(a,b)$ and $(c,d)$ are two solutions, then $a^k \neq c^k$ and $b^k \neq d^k$. Moreover, if $q > kn + 1$, then our equation has at least $n$ such solutions.
\end{lemma}
\begin{proof}
First let us assume that $q=p^m$, $p$ is odd. Let $h\colon \mathbb{F}_q^* \rightarrow \mathbb{F}_q^*$ be a power map given by $x \mapsto x^k$.  As $h$ is a group homomorphism of cyclic groups, therefore $H:= im(h)$ is a cyclic subgroup of $\mathbb{F}_q^*$ generated by $\xi^k$, where $\mathbb{F}_q^*= \gen{\xi}$. Now, since $-1  \in H$, hence $a \in H$ implies $-a \in H$. Thus order of $H$ is even, which is given by $\frac{q-1}{(k, \ q-1)}$.  Let us rearrange the elements of $H=im(h)$ as $\{a_1,\ldots,a_t\}$ such that $a_i+a_{t+1-i}=0$, for all $1\le i\le t$, where $t=\frac{q-1}{(k, \ q-1)}$. Now let $A$ be the set of required solutions of our kind. Now since $a_i^k \neq a_{j}^k$ for any $i \neq j$, hence $(a_i, a_{t+1-i}) \in A$ for every $1\leq i \leq t$. Thus, the cardinality of $A$ is at least $t$. Now let us assume that $(u,v) \in A$, which is not equal to $(a_i, a_{t+1-i})$ for any $i$. Then $u = u_1^k \in H$ for some $u_1 \in \mathbb{F}_q$, hence $u=a_i$  and $v=a_{t+1-i}$ for some $i$. Thus we get $|A| \leq t$, and hence $|A|=t$. Note that $(0,0)$ is also a solution, hence the total number of required solutions of the given equation is  $t +1$, where $t=\frac{q-1}{(k, \ q-1)}$. 

Now let $q=2^m$. Again consider $h\colon \mathbb{F}_q^* \rightarrow \mathbb{F}_q^*$  power map given by $x \mapsto x^k$.  Then $H:= im(h)=\{ a_1, \ldots, a_t\}$ is a cyclic subgroup of $\mathbb{F}_q^*$ generated by $\xi^k$, where $\mathbb{F}_q^*= \gen{\xi}$ and $t=\frac{q-1}{(k, \ q-1)}$.  Let $A$ be the set of required solutions of our kind.   As $1=-1$ in $\mathbb F_q$, therefore $a=-a$ for every $a \in \mathbb F_q$. Thus we get  $a_i+a_i=2a_i=0$ for all $a_i \in H$, $1 \leq i \leq t$. Hence $(a_i,a_i) \in A$ for all $1\leq i \leq t$, which implies $|A| \geq t $. Now let $(u,v) \in A$. Then $u=u_1^k$ for some $u_1 \in \mathbb F_q$, hence $u=a_i$ for some $1 \leq i \leq t$. As $u+v=0$, therefore we get $u=-v=v$, which implies $v=a_i$. Thus $(u,v)=(a_i,a_i)$ and we get $|A| \leq t$. As $(0,0)$ is also a solution, therefore  $|A|=t+1$.
The last part is clear, as $q > kn + 1$ implies $\frac{q-1}{(k,q-1)} \geq \frac{q-1}{k} > n $.

\end{proof}

\begin{remark}\label{zero2}
In the previous proof, for every non zero element $(u,v) \in A$, define $V_{u,v} = \{(x, y) \mid x \in h^{-1}(u), y \in  h^{-1}(v)\}$. Notice that $|V_{u,v}| \leq k^2$.  It is clear that either $V_{u,v} \cap V_{a,b}= \phi$ or $V_{u,v}=V_{a,b}$, for every $(u,v), (a,b) \in A$. Thus we get disjoint $V_i's$ as in Lemma \ref{2var1}, where $1 \leq i  \leq t$, where $t= \frac{q-1}{(k, q-1)}$. 
\end{remark}

\begin{remark}\label{solnn}
By Proposition \ref{2var1}, Lemma \ref{zero1} and Remark \ref{zero2}, we get that if $q > n^2k^{16}$, then the equation $X^k+ Y^k=\lambda$, where $\lambda \in \mathbb{F}_q$, have at least $n$, say $r \geq n$, many disjoint $V_i$, i.e. we get  $r$ many solutions  $(x_i,y_i) \in V_i \subseteq \mathbb F_q^2$, $1\leq i \leq r$,  such that $x_i^k \neq x_j^k$  and $y_i^k \neq y_j^k$ for any $1 \leq i\neq j \leq r$.
\end{remark}

\begin{cor}\label{nvar1}
Let $\lambda$ be any element in $\mathbb{F}_q$. Consider the polynomial 
$$F(X, Y, Z) = X^k + Y^k + Z^k - \lambda $$
in $\mathbb{F}_q[X,Y,Z]$.
Then,  for all $q > \mathcal C(k,n)=n^2k^{16}$, the equation $F = 0$ has at least $n$ solutions $(x_i, y_i,z_i)$, satisfying the following conditions:
\begin{enumerate}
\item  $x_i^k \neq y_i^k$ for all $ 1 \leq i \leq n$.
\item $x_i^k \neq x_j^k$ and $y_i^k \neq y_j^k$ for every $i \neq j$.
\end{enumerate}
\end{cor}
\begin{proof}
If $\lambda=0$, then we can take $Z$ to be any non-zero element, say $\alpha$, in $\mathbb F_q^*$. Now our equation becomes $X^k +Y^k =\lambda'$ where $\lambda'=(-\alpha)^k \neq 0$. Similarly, if $\lambda \neq 0$, then we can take $Z=0$, and our equation reduces to $X^k + Y^k=\lambda$. In either of the cases, by suitable choice of $Z$, we can reduce our equation to two variables, say $X^k + Y^k=\lambda'$, $\lambda' \in \mathbb F_q^*$. Now by Proposition \ref{2var1}, for all $q > n^2k^{16}$, we get at least $n$ many solutions of the equation $X^k + Y^k = \lambda'$, which satisfies our conditions. Hence we are done.
\end{proof}

\begin{prop}\label{2var2}
Let $k\geq 2$ and $m\geq 2$.  Consider the system of equations over $\mathbb F_q$, 
\begin{eqnarray*}
X_1^k+Y_1^k &=& \lambda_1,\\ \vdots   &&     \\
X_m^k+Y_m^k&=&\lambda_m,
\end{eqnarray*}
where $\lambda_i$ are distinct elements in $\mathbb{F}_q$.  Then, for all $q > 4m^2k^{16}$, the system of equations has at least $m$ solutions $(x_i, y_i) \in \mathbb{F}_q^2$, $1\leq i \leq m$, satisfying the following conditions:
\begin{enumerate}
\item $x_i^k + y_i^k= \lambda_i$.
\item  $x_i^k \neq y_i^k$ for all $ 1 \leq i \leq m$.
\item $x_i^k \neq x_j^k$ and $y_i^k \neq y_j^k$ for every $i \neq j$.
\end{enumerate}
\end{prop}
\begin{proof}
It might be possible that $\lambda_i = 0$ for some $1 \leq i \leq m$. In that case, by rearranging the system of equations, we may assume $\lambda_1 = 0$. For every $\lambda_i$, let $S_i$  be the set of solutions of equation $X_i^k +Y_i^k= \lambda_i$  and  $|S_i| = \mathcal{N}_i$. Let $V_{ir}$ be subsets of $S_i$  defined as in the proof of Proposition \ref{2var1}, where $1 \leq i \leq m$ (the index $r$ indicates that there could be several such subsets). Then for fixing $\lambda_i$, $1 \leq i \leq m$, $V_{ir}=V_{u,v}$ for some $(u,v)\in S_i$. Hence $|V_{ir}|\leq k^2$ and $V_{ir} \cap V_{is} =\phi$ for $r \neq s$. Now by Remark~\ref{solnn}, if we take $\mathcal{C}(k, m) = 4m^2k^{16}$, then for every $q> \mathcal{C}(k,m)$ we get 
$$\mathcal{N}_i > q-k^4\sqrt{q} > k^{16}( 4m^2- 2m) = k^{16}2m( 2m- 1)> 2mk^{2}+1,$$
which  implies that we get at least $2m$ many disjoint $V_{ir}$, for every $1 \leq i \leq m$. Let $t_i$ be the total number of disjoint $V_{ir}$, where $1 \leq i \leq m$ and $t_i \geq 2m$, i.e $1 \leq r \leq t_i$. 
 
Let $W_i=\mathop{\dot{\bigcup}}_{r=1}^{t_i} V_{ir}$, $1 \leq i \leq m$. Now take an element in $ W_1$, say $(a_{1}, b_{1})$ and define $A_1 = \{(a_{1}, b_{1})\}$. It is possible that, there exists an element $(c,d) \in W_2 = \mathop{\dot{\bigcup}}_{r=1}^{t_2} V_{2r}$, for some $1 \leq r \leq t_2$  such that $a_{1}^k=c^k$. Let us assume that such element $(c,d) \in V_{2r}$ exists for some $1 \leq r \leq t_2$. Now we claim that there does not exist any $(e,f) \in V_{2s}$, $s \neq r$ such that $a^k = e^k$. On contrary, let us assume that $(e,f) \in V_{2s}$ for some $s\neq r$ such that $e^k =a^k$. Since we have $c^k + d^k = \lambda_2 = e^k + f^k$ and $c^k = a^k = e^k$, therefore we get $d^k=f^k$. Thus $(e, f) \in V_{2r}$, which is not possible, as $V_{2r} \cap V_{2s}=\phi$. Thus there can exist an element $(c,d) \in V_{2r} \subseteq W_2$ for at most one $1 \leq r \leq t_2$. Similarly, there might exist an element $(e,f) \in V_{2s} \subseteq W_2$, such that $b_{1}^k=f^k$ for at most one $1 \leq s \leq t_2$. We saw that for $(a_1, b_1)$, there can exist at most two $V_{2r}$ and $V_{2s}$ such that  $a_1^k = c^k$ and $b_1^k = f^k$ where $(c,d ) \in V_{2r}$ and $(e,f)\in V_{2s}$. Thus we are left with at least $2m -2 >0$ many disjoint $V_{2r}$, from which we can choose any element, say $(a_2, b_2)$, which will satisfy the condition $a_1^k \neq a_2^k$ and $b_1^k \neq b_2^k$. Set $A_2=\{(a_1, b_1), (a_2, b_2) \}$. As $W_3 = \mathop{\dot{\cup}}_{r=1}^{t_3} V_{3r}$ and by definition of $V_{3r}$, we can see by same process that for each $(a_i,b_i)\in A_2$ there can exist at most two  $V_{2r}$ and $V_{2s}$ such that  $a_i^k = c^k$ and $b_i^k = f^k$ where $(c,d ) \in V_{2r}$ and $(e,f)\in V_{2s}$. Thus we get at least $2m-2\cdot2>0$ many disjoint $V_{3r}$'s, from which we can choose any element, say $(a_3, b_3)$ such that $a_i^k \neq a_3^k$ and $b_i^ \neq b_3^k$, $1 \leq i \leq 2$. Set $A_3=\{(a_i,b_i) \mid a_i^k \neq a_j^k, b_i^k \neq b_j^k, 1 \leq i \neq j \leq 3  \}$. Since $2m-2(m-1)>0$, therefore continuing like this we will get a set $A_m=\{(a_i,b_i) \mid a_i^k \neq a_j^k, b_i^k \neq b_j^k, 1 \leq i \neq j \leq m  \}$. By construction, it is clear that every element of $A_m$ will satisfy the required conditions and the cardinality of $A_m$ is $m$. 
\end{proof}

\begin{cor}\label{nvar2}
Let $k \geq 2$ and $m \geq 2$, and the characteristic of $\mathbb{F}_q$ be not $2$. Consider the system of equations over $\mathbb F_q$, 
\begin{eqnarray*}
X_1^k + Y_1^k + Z_1^k & = &\lambda_1,\\  \vdots  && \\
X_m^k + Y_m^k + Z_m^k &=& \lambda_m,
\end{eqnarray*}
where $\lambda_i$ are distinct elements in $\mathbb{F}_q$. Then, for all $q > 4{m^2}k^{16}$, the system of equations has at least $m$ solutions $(x_i, y_i, z_i) \in \mathbb{F}_q^2$, $1\leq i \leq m$, satisfying the following conditions:
\begin{enumerate}
\item $x_i^k + y_i^k + z_i^k = \lambda_i$.
\item  $x_i^k \neq y_i^k$ for all $ 1 \leq i \leq m$.
\item $x_i^k \neq x_j^k$ and $y_i^k \neq y_j^k$ for every $i \neq j$.
\end{enumerate}
\end{cor}
\begin{proof}
As in previous Proposition, by suitable choice of $Z_i$, say $\alpha_i$ in $\mathbb F_q$, we get a new system of equations
\begin{eqnarray*}
X_1^k + Y_1^k &=& \lambda_1', \\ \vdots  &&   \\
X_m^k + Y_m^k &=& \lambda_m',
\end{eqnarray*}
where $\lambda_i' = \lambda - \alpha_i^k \neq 0$. Now by Proposition~\ref{2var2}, for all $q > 4{m^2}k^{16}$, we get at least $m$ many required solutions of our system of equations, which satisfies our conditions. Hence we are done.
\end{proof}

\begin{remark}\label{2soln}
Let $q=p^m$, $m\geq 1$ where $p$ is an odd prime. Then the equation $X_1^{k} + X_2^{k} = 1$, where $k = q-1$, has only two solutions $(x_1, x_2)$ and $(y_1, y_2)$ with the property $x_1^k \neq y_1^k$ and $x_2^k \neq y_2^k$. For this, note that $\mathbb{F}_q^*$ is a cyclic group of order $q-1$. Hence, if $x \in \mathbb{F}_q$, then $x^{q-1}\in \{0, 1\}$. Thus we get only two solutions, $(a,0)$ and $(0,a)$, where $a \in \mathbb{F}_q^*$. In the statement of our theorem, we need $q$ large enough and the proof relies on having large enough solutions over the base field. This example indicates the role of the same. 
\end{remark}

%%%%%%%%%%%%%%%%%%%%%%%%%%%%

\section{Similarity classes in $T_n(\mathbb F_q)$ and reduction to Indecomposable matrices}\label{section-similarity-classes}
To solve the Waring problem it is enough to do it for a similarity/conjugacy class representative. Writing the conjugacy classes of $T_n(\mathbb F_q), B_n(\mathbb F_q), U_n(\mathbb F_q)$ and $N_n(\mathbb F_q)$ is a difficult problem and is yet to be understood fully (see~\cite{KO05, MMH, VA91, VA92, VAGO17}). We briefly recall the graph representation of matrices in $M_n(\mathbb{F}_q)$, the set of all $n\times n$ matrices, see \cite{MMH}, which we use in our work. Every $A = [a_{ij}] \in M_n$ is the adjacency matrix of a directed graph $G_A = (V_A, E_A)$ with a weight function $w_A \colon [n] \times [n]  \rightarrow \mathbb{F}_q^*$, whose support is $E_A$ and $[n]=\{1,2,\ldots, n\}$. More precisely,
 $$V_A =[n];\ \ E_A =\{(i,j)\in [n] \times [n] \mid a_{ij}  \neq 0\};\ \ w_A(i,j) = a_{ij}.$$
Each element of $V_A$ (respectively $E_A$) is called a vertex (respectively an arc) of the graph $G_A$. Each arc $(i,j) \in E_A$ is visualized as $i \rightarrow j$, in which $i$ (respectively $j$) is called the tail (respectively the head) of the arc $(i, j)$, and $w_A(i, j)$ is called the weight of the arc $(i, j)$. Call $G_A = (V_A, E_A)$ the graph of $A$, and $\Tilde{G}_A = (V_A, E_A, w_A)$ the weighted graph of $A$. When $A \in N_n$, the graph of $A$ is simple and it consists of some arcs $(i,j) \in [n] \times [n]$ with $i < j$. A partition of $[n]$ is given by $S_1 \cup \cdots \cup S_m$ where partition subsets $S_i \neq \phi$. For the uniqueness of expression, we assume that the minimal elements of $S_1, \ldots, S_m$ are in ascending order, and write the partition as $S'_1 | \cdots | S'_m$, where the elements of $S'_i$ are in ascending order. For example, the partition of $\{4, 1\} \cup \{6, 5\}\cup \{3, 2\} $ of $[6]$ will be expressed as $14|23|56$. Let us understand this by an example: 

\begin{example}
    Let $A$ be the matrix \[ \begin{bmatrix}  0 & 1 & 1  & 0\\
   & 0 & 0  & 0\\     &   & 0  & 1\\     &   &    & 0\\ \end{bmatrix}, \]
then it will be presented as $12|34:13$ and the graph associated with $A$ will be
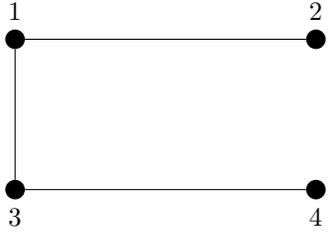
\begin{center}
\begin{tikzpicture}

\centering
	\vertex (1) at (-2, -2) [label=above:${1}$] {}; 
	\vertex (2) at (2, -2) [label=above:${2}$] {};
	\vertex (3) at (-2, -4) [label=below:${3}$] {}; 
	\vertex (4) at (2, -4) [label=below:${4}$] {};
\path
	   % Note that the word "path" here isn't used in the graph-theory sense; the \path command
	   % is always used prior to the list of edges; here, coincidentally, they do form an actual path.
(1) edge (3)
(1) edge (2)
(3) edge (4);
\end{tikzpicture}.
    
\end{center}
In the graph presentation of $A$, $13$ is known to be an extra arc of the graph, see \cite{MMH}.
\end{example}

The \textit{principal} submatrices of a matrix are the matrix itself, and those submatrices obtained from it by repeatedly deleting out a row and a column of the same index. A matrix $A \in T_n(\mathbb{F}_q)$ is said to be \textit{indecomposable} if $A$ is $B_n$-similar to $A_1 \oplus A_2$ implies that $A_1$ or $A_2$ has size $0 \times 0$, see last paragraph of~\cite[Section 3]{CX16} for details. Another definition used in the literature is that a matrix $A \in M_n$ is said to be indecomposable if no permutation matrix $P \in M_n$ satisfies that $PAP^T$ can be written as a direct sum of two proper principal submatrices. Both these definitions of indecomposable matrices are equivalent when referring to the indecomposable Belitski\v{i}'s form as mentioned in~\cite[Section 2.3]{MMH}. We mention the following result from~\cite[Corollary 4.3]{MMH}.
\begin{prop}
Let $A \in N_n(\mathbb{F}_q)$ be Belitski\v{i}’s canonical form. Then, $A$ is indecomposable if and only if the graph $G_A = ([n], E_A)$ is connected. 
\end{prop}

Let $A = (a_{ij} ) \in T_n(\mathbb{F}_q)$. The entries of $A$ are ordered as follows
$$a_{ij}  \preccurlyeq a_{i'j'} \ \ if \ \ i > i',\  or \ \  i = i' \ \ but \ \  j < j'.$$
 In general, we can reduce the matrix $A$ to its canonical form inductively ~\cite{CX16}. We will use the same technique, include this discussion here, as it is crucial for our work.
%It is clear that, if $A$ and $B$ are $B_n$-conjugate, then $a_{ii}=b_{ii}$ for all $1 \leq i\leq n$.

\begin{lemma}\label{red1}
Let $A\in T_n(\mathbb{F}_q)$ such that $a_{ll} \neq a_{rr}$, for some fixed $1 \leq r, l \leq n$. If $A'$ is the (Belitski\v{i}’s) canonical form of $A$ under $B_n$-similarity, then $a'_{lr} = 0 = a'_{rl}$.
\end{lemma}
\begin{proof}
Without loss of generality, we may assume that $r > l$ and  $l$ is the largest positive integer such that $a_{rr}\neq a_{ll}$. As $r>l$, therefore $a'_{rl}=0$.

Suppose that all the entries before the $(l, r)$-position have already been reduced and after $t$-th step we obtain the upper triangular matrix $A^{(t)}$. We denote the admissible transformation group by $B^{(t)}_n$ which preserves all reduced entries before the $(l, r)$-position, that is, 
$$B^{(t)}_n=\{ S \in B_n \mid A^{(t)}S  \equiv_{\preccurlyeq (l,r-1)} SA^{(t)}\}.$$
Now, we reduce the entry $a^{(t)}_{lr}$ under the admissible transformations in $B^{(t)}_n$. Let $S^{(t)} = (s^{(t)}_{ij} ) \in B^{(t)}_n$ and $A^{(t+1)} = S^{{(t)}^{-1}}A^{(t)}S^{(t)}$. We have $A^{(t)}S^{(t)} = S^{(t)}A^{(t+1)}$.  If we compare $lr$-entry of both sides, then we get
$$a^{(t)}_{l,l}s^{(t)}_{l,r}+a^{(t)}_{l,l+1}s^{(t)}_{l+1,r}  +\cdots+ a^{(t)}_{lr} s^{(t)}_{rr} = s^{(t)}_{ll} a^{(t+1)}_{lr} + s^{(t)}_{l,l+1}a^{(t+1)}_{l+1,r} + \cdots+ s^{(t)}_{l,r}a^{(t+1)}_{r,r}.$$
Note that $(i, r) \preccurlyeq (l, r)$ for $i = l + 1, \ldots, r - 1$. Then we have $a^{(t+1)}_{ir} = a^{(t)}_{ir}$ for $i =l + 1, \ldots, r - 1$. The previous equation changes into
$$a^{(t)}_{l,l}s^{(t)}_{l,r} + a^{(t)}_{l,l+1}s^{(t)}_{l+1,r}  + \cdots + a^{(t)}_{lr} s^{(t)}_{rr} = s^{(t)}_{ll} a^{(t+1)}_{lr} + s^{(t)}_{l,l+1}a^{(t)}_{l+1,r} + \cdots+ s^{(t)}_{l,r-1}a^{(t)}_{r-1,r}+ s^{(t)}_{l,r}a^{(t)}_{rr}$$
which gives 
$$(a^{(t)}_{l,l}-a^{(t)}_{rr})s^{(t)}_{l,r}+\alpha = s^{(t)}_{ll} a^{(t+1)}_{lr},$$
where $\alpha=(a^{(t)}_{l,l+1}s^{(t)}_{l+1,r}  +\cdots+ a^{(t)}_{lr} s^{(t)}_{rr}) - (s^{(t)}_{l,l+1}a^{(t)}_{l+1,r} + \cdots+ s^{(t)}_{l,r-1}a^{(t)}_{r-1,r}). $
As $a^{(t)}_{ll} \neq a^{(t)}_{rr}$, therefore there exist $S^{(t)} \in B^{(t)}_n$ such that $s^{(t)}_{l,r} = (s^{(t)}_{ll} a^{(t+1)}_{lr}-\alpha)(a^{(t)}_{l,l}-a^{(t)}_{rr})^{-1}$. Thus we get an $S^{(t)} \in B^{(t)}_n$ such that $a^{(t+1)}_{lr}=0$ in $A^{(t+1)} = S^{{(t)}^{-1}}A^{(t)}S^{(t)}$. Hence if $a_{ll}\neq a_{rr}$, then $(A^{(t+1)})_{lr}=0$, which implies the $(l,r)$-entry of the canonical form of $A$ is zero.         
\end{proof}

\begin{cor}\label{dia}
Let $A \in  T_n(\mathbb{F}_q)$ such that all diagonal entries of $A$ are distinct. Then Belitski\v{i}’s canonical form of $A$ under $B_n$-similarity is a diagonal matrix.
\end{cor}

Thus by Lemma~\ref{red1}, we get that the matrix which has at least two different eigenvalues, is a decomposable matrix. Let $A \in T_n(\mathbb{F}_q)$ and $m$ be a positive integer. Then the $ij$-th entry of $A^m$ is given by $\sum \big( \prod_{k=1}^m A_{i_{k-1} i_k} \big)$, where the sum is over all non decreasing $(m+1)$-tuples  $i = i_0 \le i_1 \le \ldots \le i_m = j$ starting at $i$ and ending at $j$. 

\begin{lemma}\label{reduction}
Let $ A \in T_n(\mathbb{F}_q)$ be a $k$-th  power. Let $B \in T_{n+1}(\mathbb{F}_q)$ be obtained in such a way that $A$ is the principal submatrix of $B$ obtained by eliminating $l$-th row and $l$-th column. Further assume that $b_{il}=b_{li}=0$ for all $i \neq l$ with $b_{ll}=x^k$ for some $x \in \mathbb{F}_q^*$. Then, $B$ is also a $k$-th  power.
\end{lemma}
\begin{proof}
If $l=1$ or $l=n$, then $B= [x^k]  \oplus A$ and $B=A \oplus [x^k]$, respectively and hence $B$ is a $k$-th power. Now we may assume $ 1< l < n$. Let $A=P^k$ for some $P =(P_{ij})\in T_n(\mathbb{F}_q)$. We know that the $ij$-th entry of $A$ is given by  $\sum \big( \prod_{m=1}^k P_{i_{m-1} i_m} \big)$ where the sum is over all non decreasing $(k+1)$-tuples $i = i_0 \le i_1 \le \ldots \le i_k = j$ starting at $i$ and ending at $j$. Now the matrix $B$ is as follows, 
\[ b_{ij} = \begin{cases} 
   x^k & i=j=l \\
   0 & i=l \  or \  j=l, \ \rm{but} \  i\neq j\\
   a_{\sigma(i)\sigma(j)} & \rm{otherwise},   
\end{cases}
\] 
where $\sigma \colon [n+1] \rightarrow [n]$ is a map sending $i$ to $i-1$ for $i > l$ and fixing rest of the elements, and $[n]=\{1,2, \ldots, n\}$.
Similarly, we define a matrix $P'\in T_{n+1}(\mathbb F_q)$ with entries
 \[ p'_{ij} = \begin{cases} 
     x & i=j=l \\
     0 & i=l \  or \  j=l, \ \rm{but} \  i\neq j\\
     P_{\sigma(i)\sigma(j)} & \rm{otherwise}.   
    \end{cases}
\]
We claim that $P'^k=B$. For simplicity we denote $X=(X_{ij}):=P'^k$. It is clear that, $X_{il}=X_{li}=0= B_{il}= B_{li}$, when $i \neq l$ and $X_{ll} = x^k$. So now we can assume that $i\neq l $ and $j \neq l$. We know
\begin{eqnarray*}
X_{ij} & =& \sum \big( \prod_{m=1}^k P'_{i_{m-1} i_m} \big) = \sum \big( \prod_{m=1}^k P_{\sigma(i_{m-1}) \sigma(i_m)} \big)\\
    & =& \sum \big( P_{\sigma(i_0) \sigma(i_1)}P_{\sigma(i_1) \sigma(i_2)} \cdots P_{\sigma(i_{k-1}) \sigma(i_k) }\big),
 \end{eqnarray*}
where $i_0=i$ and $i_k=j$. Now, 
\begin{enumerate}
\item[{Case (i)}] if $i , j < l$, then $X_{ij}=\sum P_{i_0 i_1}P_{i_1 i_2} \cdots P_{i_{k-1}i_k } = A_{ij}=A_{\sigma(i)\sigma(j)}=B_{ij}$.
\item[{Case (ii)}]  If $i , j > l$, then 
$$
X_{ij} =\sum  P_{i_{0}-1, i_{1}-1}P_{i_{1}-1, i_{2}-1} \cdots P_{i_{k-1}-1, i_{k}-1 }  = A_{i-1,j-1} =A_{\sigma(i)\sigma(j)} =B_{ij}.
$$
\item[{Case(iii)}]  Finally consider $i <l <j$. Let $i_{r}-1<l<i_r$ for some $r$. Then 
$$
X_{ij} =\sum  P_{i_{0}, i_{1}}\cdots P_{i_{r-1}, i_{r}-1} \cdots P_{i_{k-1}-1, i_{k}-1 } = A_{i,j-1}  =A_{\sigma(i)\sigma(j)} =B_{ij}.
$$
\end{enumerate}
Hence $X_{ij}=B_{ij}$ for all $1 \leq i,j\leq n$.   
\end{proof}

\begin{example} In the previous Lemma, if we take $A = \begin{bmatrix}   a & b & c  \\     0 & d & e\\     0 & 0 & f \end{bmatrix}$, and further suppose it is a $k$-th   power. Then $B$ can be equal to one of the following matrices, for some $x$, all of which are $k$-th   powers, 
\[
  \begin{bmatrix}     x^k & 0 & 0 & 0\\     0 & a & b & c\\     0 & 0 & d & e\\
    0 & 0 & 0 & f \end{bmatrix},
  \begin{bmatrix}     a & 0 & b & c  \\     0 & x^k & 0 & 0\\     0 & 0 & d & e\\     0 & 0 & 0 & f \end{bmatrix},
  \begin{bmatrix}     a & b & 0 & c\\     0 & d & 0 & e\\     0 & 0 & x^k & 0\\
    0 & 0 & 0 & f \end{bmatrix},
\begin{bmatrix}     a  & b & c & 0 \\     0 & d & e & 0\\     0 & 0 & f & 0 \\
    0  & 0 & 0 & x^k \end{bmatrix}.
\]
\end{example} 

\begin{remark}\label{ind}
Let $A$ be a decomposable matrix in $T_n(\mathbb{F}_q)$. Then there are at most $n-2$ non-zero, non-diagonal entries in $A^{\infty}$, where $A^{\infty}$ is Belitski\v{i}’s canonical form of $A$. We will use the Lemma~\ref{red1} and Lemma~\ref{reduction}, to write the decomposable matrices in $T_n(\mathbb{F}_q)$, $n \leq 6$ as a sum of two $k$-th powers using the technique developed in the next Section. This is explained at the end of Section \ref{section-proof-theorem}. 
\end{remark}

%%%%%%%%%%%%%%%%%%%%%%%
\section{Some useful lemmas}\label{section-proof-theorem}

Let $C \in T_n(\mathbb{F}_q)$ be a matrix. If $C$ is a sum of $r $ many $k$-th   powers, then by equating the diagonals, we see that equations $c_{ii} = X_1^k + \cdots + X_r^k$ must have a solution over $\mathbb{F}_q$. Thus, if $C \in T_n(\mathbb{F}_q)$ is a $k$-th   power, then $c_{ii}$ must be a $k$-th   power for every $1 \leq i \leq n$. 
The proof of the following lemma follows from the Corollary \ref{dia} but we include an alternate proof here of which ideas we use in the later part of this paper too. 

\begin{lemma}\label{mult}
Let $C =(c_{ij}) \in T_n(\mathbb{F}_q)$ be a matrix, whose diagonal entries are distinct $k$-th   powers. Then $C = A^k$, for some $A \in T_n(\mathbb{F}_q)$.
\end{lemma}
\begin{proof}
Let $A$ be a matrix in $T_n(\mathbb{F}_q)$. Then the matrix $A^k$ will look like,
\[
\begin{bmatrix}
a_{11}^k & a_{12}f( a_{11},a_{22})+\alpha_{12}  & \cdots  & a_{1,n} f( a_{11},a_{nn} )+\alpha_{1n} \\
0  & a_{22}^k  & \cdots &   a_{2n}f( a_{22},a_{nn} )+\alpha_{2n} \\
0 & 0   & \cdots &   a_{3n}f( a_{33},a_{nn} )+\alpha_{3n} \\
\vdots    & \vdots  & \ddots   & \vdots  \\
0 & 0 & \cdots &   a_{n-1,n}f(a_{n-1,n-1},a_{nn} )+\alpha_{n-1,n} \\
0 & 0    &  \cdots &   a_{nn}^k 
\end{bmatrix},
\]
where $f(a_{ii}, a_{jj}) = a_{ii}^{k-1}+ a_{ii}^{k-2} a_{jj}+\cdots +  a_{ii}a_{jj}^{k-2} + a_{jj}^{k-1}$, and $\alpha_{rs}$ is a polynomial in $a_{ij}$, $r \leq i,j \leq s$ and  ${a_{rs}}$ is not a term of $\alpha_{rs}$. Now we compare coefficients of $A^k$ and $C$. As $a_{ii}^k=c_{ii}$, and $c_{ii}$ are distinct, therefore by Lemma \ref{homeqn}, $f( a_{ii}, a_{jj})$ is non zero for every $1 \leq i,j \leq n$, where $i\neq j$. Thus $a_{ij}= (c_{ij} - \alpha_{ij}) f(a_{ii}, a_{jj})^{-1}$. Hence $C=A^k$, for some $A \in T_n(\mathbb{F}_q)$.  
\end{proof}

\begin{lemma}\label{power}
Let $C \in T_n(\mathbb{F}_q)$ such that diagonal entries are $k$-th   powers and $c_{rs}c_{st}=0$ for all $r < s < t$. If for every $c_{ij} \neq 0$, we have $c_{ii} \neq c_{jj}$ then $C=A^k$, for some $A \in T_n(\mathbb{F}_q)$. 
\end{lemma}
\begin{proof}
Let $C^{\infty}$ be Belitski\v{i}'s canonical form of $C$.   Consider a matrix $A$ such that $a_{ii}^k=c_{ii}$. If $c_{rs } \neq 0$, then by hypothesis  we get $c_{rr} \neq c_{ss}$, which implies that $f(a_{rr}, a_{ss}) \neq 0$. Thus we can define the entries $a_{rs}$ of $A$ as follows       
\[   
a_{rs} = 
    \begin{cases}
       (c_{rs}-m_{rs})f(a_{rr},a_{ss})^{-1} &\quad c_{rs} \neq 0,\\
       0 &\quad c_{rs}  = 0.
    \end{cases}
\]
 By the given hypothesis and Lemma  \ref{zero}, we get that if $c_{ij}=0$ for some $i \neq j$, then $(C^k)_{ij}=0$, for every $k>0$. As $c_{rs}c_{st}=0$ for all $r<s<t$, therefore by construction of $A$, it is clear that $a_{rs}a_{st}=0$ for all $r<s<t$. Then by lemma \ref{mult}, we can see that $C=A^k$.    
\end{proof}

\begin{lemma}\label{row}
Let $C \in T_n(\mathbb{F}_q)$ be a matrix whose diagonal entries are  $k$-th   powers. If  $c_{nn}\neq c_{ii}$, for all $1 \leq i \leq n-1$  and $c_{ij}=0$ for all $1 \leq i<j \leq n-1$, then $C=A^k$, for some $A \in T_n(\mathbb{F}_q)$.
\end{lemma}
\begin{proof}
By the given hypothesis $c_{rs}c_{st}=0$ for every $r<s<t$. As $c_{ii}\neq c_{nn}$ for all $i \neq n$ and $c_{ij}=0$ for all  $1 \leq i<j \leq n-1$. Thus we get $c_{ij}=0$, whenever $c_{ii}=c_{jj}$. Therefore by previous lemma $C=A^k$ for some $A \in T_n(\mathbb{F}_q)$.
\end{proof}

In the remaining part of this section, 
\begin{quote}
{\bf let us assume that  for any $\alpha \in \mathbb{F}_q$ the equation $X_1^k + X_2^k = \alpha$ has two solutions in $\mathbb{F}_q\times \mathbb F_q$, say $(x_1, y_1)$, $(x_2, y_2)$ such that $x_1^k \neq x_2^k$ and $y_1^k \neq y_2^k$}. 
\end{quote}
Now, we will see if we can write every element of $T_n(\mathbb{F}_q)$ as a sum of two $k$-th   powers with the above assumption.  In this section, we will talk about the matrices $C \in T_n(\mathbb F_q)$ such that $c_{ii}=c_{jj}$ for all $i \neq j$.  It is easy to see that if $C=A \oplus B$, then $C$ is a sum of two $k$-th powers if and only if both $A$ and $B$ are a sum of two $k$-th powers. Thus it is sufficient to look at the matrices which cannot be written as a direct sum of two matrices.

\begin{lemma}\label{samerow}
Let $C \in T_n(\mathbb{F}_q)$ be a matrix whose non-zero,  non-diagonal entries are either in the same row or same column, and the diagonal entries are the same. Then, for $q > k^{16}$ we can write $C=X^k + Y^k$ for some $X, Y \in T_n(\mathbb{F}_q)$.
\end{lemma}
\begin{proof} Without loss of generality, let us assume that all non-zero,  non-diagonal entries of $C$ are in the first row. Let there are $m$ such entries, say $c_{1j_1}, c_{1j_2}, \ldots, c_{1j_m}$ such that $j_1 < j_2< \ldots < j_m$. As all the diagonal entries are same, therefore by Remark \ref{solnn}, for every $q > 4k^{16}$, there exist at least two solutions $(x_1,y_1), (x_2,y_2) \in \mathbb F_q^2$ such that we can write $c_{ii}=x_{1}^k + {y}_{1}^k = x_{2}^k+{y}_{2}^k$,  such that $x_{1}^k \neq x_{2}^k$ and ${y}_{1}^k \neq {y}_{2}^k$, for every $1 \leq i\leq n $.  We will write $c_{11} = x_{1}^k + {y}_{1}^k$ and $c_{ii}=x_{2}^k+{y}_{2}^k$, where $i \neq 1$. Let $A$ be a matrix having $a_{11}=x_1^k$, $a_{ii}=x_2^k$ and $a_{ij}=c_{ij}$. Similarly let $B$ be a diagonal matrix having  $b_{11}=y_1^k$ and $b_{ii}=y_2^k$. It is clear that $C=A+B$. Since $a_{rs}a_{st}=0=b_{rs}b_{st}$ for all $r<s<t$, and hence by Lemma \ref{power} $A$ and $B$ are $k$-th powers. Thus we are done.   Similarly, we can prove the result when non-zero non-diagonal entries are in the same column.       
\end{proof}

\begin{lemma}\label{norow}
Let $C \in T_n(\mathbb{F}_q)$ be a matrix whose non-zero, non-diagonal entries are in different rows and different columns, and all the diagonal entries are the same. Then for all  $q > k^{16}$, we have $C= A^k + B^k$ for some $A, B \in T_n(\mathbb{F}_q)$. 
\end{lemma}
\begin{proof}
By Proposition~\ref{2var} we know that for all $q>k^{16}$, there exist $(x_1, y_1), (x_2, y_2) \in \mathbb F_q^2$ such that we can write $c_{ii}=x_1^k + y_1^k = x_2^k + y_2^k$ and $x_1^k \neq x_2^k$, $y_1^k \neq y_2^k$. Let $C$ be a matrix in $T_n(\mathbb F_q)$ such that $c_{i_1j_1}, \ldots, c_{i_m j_m}$ are non diagonal and non zero entries, where $i_1\leq  \ldots \leq i_m$. Hence by given hypothesis, we get $i_r \neq i_s$, $j_r \neq j_s$, $1 \leq r,s \leq m$. Now first  assume that $c_{rs}c_{st}=0$ for all $r<s<t$. Then clearly $j_r \neq i_s $, for all $1 \leq r \neq s \leq m$. Hence for every $c_{i_rj_r} \neq 0$, we write $c_{i_ri_r}=x_1^k + y_1^k$ and $c_{j_rj_r} = x_2^k + y_2^k$. For the rest of the diagonal entries, we can write $c_{ii} = x_1^k + y_1^k$. As we wrote down every $c_{ii}$ in such a way that it satisfies the condition if $c_{ij}\neq 0$, then $c_{ii} = x_r^k + y_r^k$ and $c_{jj}=x_s^k + y_s^k$ such that $x_r^k \neq x_s^k$ and $y_r^k \neq y_s^k$, where $1 \leq r \neq s \leq 2$. Now, let $A $ be the matrix with diagonal entries $a_{ii}=x_r^k$ and $a_{ij}=c_{ij}$, and $B$ be a diagonal matrix having $b_{ii}=y_{r}$ such that $c_{ii} = x_r^k + y_r^k$, which also satisfies the above-mentioned condition. It is clear that $B$ is $k$-th power in $T_n(\mathbb F_q)$.  As diagonal entries of $A$ are $k$-th powers and $a_{rs}a_{st}=0$ for all $r <s<t$, moreover if $a_{ij} \neq 0$, then $a_{ii} \neq a_{jj}$, therefore by Lemma \ref{power}, we get that $A$ is a $k$-th power. Thus $C = A + B$, where $A$ and $B$ are $k$-th powers in $T_n(\mathbb F_q)$.

Now let us assume that $c_{rs}c_{st} \neq 0$ for some $r <s<t$. Let $$ \{ c_{i_1j_1}, \ldots, c_{i_{r_1} j_{r_1}},c_{i_{r_1+1}j_{r_1+1}}, \ldots, c_{i_{r_1+r_2}j_{r_1+ r_2}}, \ldots, c_{i_{r_1+ \ldots r_{l-1} +1}j_{r_1 + \ldots r_{l-1}+1}}, \ldots, c_{i_{r_1 + \ldots +r_l}j_{r_1+ \ldots +r_l}} \}, $$
 be the ordered set of non-diagonal non zero entries of $C$ such that $j_{r-1} = i_r$ for all $2 \leq r \leq r_1$, \ $r_1+2 \leq r \leq r_1+r_2$, \  \ldots, \  $r_1+ \ldots +r_{l-1}+2 \leq r \leq r_1+\ldots+r_l$.  Now let us write 
\begin{eqnarray*}
     c_{i_1i_1} &=& x_1^k+y_1^k = c_{i_{r_1+1}i_{r_1+1}}= \ldots= c_{i_{r_1+ \ldots r_{l-1} +1}i_{r_1 + \ldots r_{l-1}+1}},\\
     c_{j_{m}j_{m}} &=& \begin{cases} 
   x_1^k+y_1^k & \text{if $1 \leq m \leq r_1$ is even},\\ 
   x_2^k+y_2^k & \text{if $1 \leq m \leq r_1$ is odd},
\end{cases}\\
\vdots \\
c_{j_{r_1+\ldots+r_{l-1}+m}j_{r_1+\ldots+r_{l-1}+m}} &=& \begin{cases} 
   x_1^k+y_1^k & \text{if $1 \leq m \leq r_l$ is even},\\ 
   x_2^k+y_2^k & \text{if $1 \leq m \leq r_l$ is odd}
\end{cases}\\
 \end{eqnarray*}
and we write rest of the $c_{ii}$ as $x_1^k + y_1^k$. Now let $A $ and $B$ be  matrices such that 
\begin{eqnarray*}
a_{i_1i_1} &=& x_1^k= a_{i_{r_1+1}i_{r_1+1}}= \ldots= a_{i_{r_1+ \ldots r_{l-1} +1}i_{r_1 + \ldots r_{l-1}+1}},\\
b_{i_1i_1} &=&  y_1^k=b_{i_{r_1+1}i_{r_1+1}}= \ldots= b_{i_{r_1+ \ldots r_{l-1} +1}i_{r_1 + \ldots r_{l-1}+1}},\\
a_{j_{m}j_{m}} &=& \begin{cases} 
x_1^k &  \hspace{1cm} \text{if $1 \leq m \leq r_1$ is even},\\ 
   x_2^k & \hspace{1cm} \text{if $1 \leq m \leq r_1$ is odd},
\end{cases}\\ 
a_{i_{m}j_{m}} &=&  \begin{cases} 
0 & \hspace{.5cm} \text{if $1 \leq m \leq r_1$ is even},\\ 
c_{i_{m} j_{m}}  & \hspace{.5cm} \text{if $1 \leq m \leq r_1$ is odd},
\end{cases}\\
b_{i_mj_m} &=& \begin{cases}
c_{i_mj_m} & \hspace{.4cm} \text{ if $1 \leq m \leq r_1$ is even},\\
0 & \hspace{.5cm} \text{if $1 \leq m \leq r_1$ is odd},
\end{cases}\\
\vdots \\
a_{j_{r_1+\ldots+r_{l-1}+m}j_{r_1+\ldots+r_{l-1}+m}} &=& \begin{cases} 
x_1^k & \hspace{3.6cm} \text{if $1 \leq m \leq r_l$ is even},\\ 
x_2^k & \hspace{3.6cm} \text{if $1 \leq m \leq r_l$ is odd},
\end{cases}\\ 
a_{i_{r_1+\ldots+r_{l-1}+m}j_{r_1+\ldots+r_{l-1}+m}} &=&  \begin{cases} 
0 & \text{if $1 \leq m \leq r_l$ is even},\\ 
c_{i_{r_1+\ldots+r_{l-1}+m}j_{r_1+\ldots+r_{l-1}+m}}  & \text{if $1 \leq m \leq r_l$ is odd},
\end{cases}\\
\end{eqnarray*}
\begin{eqnarray*}
b_{j_{r_1+\ldots+r_{l-1}+m}j_{r_1+\ldots+r_{l-1}+m}} &=& \begin{cases} 
y_1^k & \hspace{3.6cm} \text{if $1 \leq m \leq r_l$ is even},\\ 
y_2^k & \hspace{3.6cm} \text{if $1 \leq m \leq r_l$ is odd},
\end{cases}\\ 
b_{i_{r_1+\ldots+r_{l-1}+m}j_{r_1+\ldots+r_{l-1}+m}} &=&  \begin{cases} 
c_{i_{r_1+\ldots+r_{l-1}+m}j_{r_1+\ldots+r_{l-1}+m}} & \text{if $1 \leq m \leq r_l$ is even},\\ 
0  & \text{if $1 \leq m \leq r_l$ is odd},
\end{cases}\\
 \end{eqnarray*}
rest of $a_{ii}=x_1^k$ and $b_{ii}=y_1^k$. Note that we constructed $A$ and $B$ in such a way that, $a_{rs} a_{st} = 0 = b_{rs} b_{st}$ for all $r<s<t$. Also, if $a_{ij} \neq 0$, then $a_{ii} \neq a_{jj}$ and similarly if $b_{ij}\neq 0$, then $b_{ii} \neq b_{jj}$. Thus $A$ and $B$ satisfy the conditions of Lemma \ref{power}, which implies $A$ and $B$ are $k$-th powers. By our construction, it is clear that $C=A+B$. This completes the proof.
\end{proof}

%As $c_{i_1j_1} \neq 0$, therefore we write $c_{i_1i_1} = x_1^k+y_1^k$ and  $c_{j_1j_1} = x_2^k+y_2^k$. If $j_1=i_r$ for some $1 \leq r \leq t$, then $c_{i_1j_1}c_{i_rj_r}\neq 0$, and hence $c_{i_ri_r}=x_2^k+y_2^k$ and we write $c_{j_rj_r}=x_1^k+y_1^k$. Now again if 

Other than the above-mentioned classes, we present the expression of indecomposable matrices in Table \ref{table1},\ref{table2} and \ref{table3} for $n \leq 6$, in which every matrix is written as a sum of two $k$-th  powers. To write a matrix $C$ into sum of two $k$-th  powers, we write $C = A + B$  such that  $c_{ii} = x_{r}^k + y_{r}^k$, where $x_1^k \neq x_2^k$ and $y_1^k \neq y_2^k$, $1 \leq i \leq n$, $r=1,2$ and diagonal entries of $A$ and $B$ are $x_{r}$ and $y_{r}$ respectively. Now, $a_{ij} = c_{ij}$ or $b_{ij} = c_{ij}$ for $i \neq j$ such that $a_{rs}a_{st} = 0 = b_{rs}b_{st}$ for $r < s < t$, and  if $a_{ij} \neq  0$,  then $a_{ii}^k \neq a_{jj}^k$ and similarly if $b_{ij} \neq  0$,  then $b_{ii}^k \neq b_{jj}^k$. This technique works on matrices up to size $6 \times 6$. We used this technique to write expressions in Table \ref{table1},\ref{table2} and \ref{table3} for $n \leq 6$. But this technique does not apply on a matrix of size $7 \times 7$, which is given by, 

 \[ C= \begin{bmatrix}
    0 & 1 & 1 & 0 & 0 & 0 & 0\\
     & 0 & 0 & 0 & 0 & 1 & 0\\
     &  & 0 & 1 & 0 & 0  & 0\\    
     &  &  & 0 & 1 & 1  & 0\\
     &  &  &  & 0 & 0  & 0\\
     &  &  &  &  & 0  & 1\\
     &  &  &  &  &   & 0\\
\end{bmatrix}. \]
If we try to write $C = A^k + B^k$ using above mentioned technique, then as $c_{12}, c_{34}, c_{45}$ and $c_{67}$ are non zero, therefore by Lemma \ref{2*2}  we get that $a_{11}^k \neq a_{22}^k, \ a_{33}^k\neq a_{44}^k, \ a_{44}^k \neq a_{55}^k$ and $a_{66}^k \neq a_{77}^k$. So the diagonal entries of $A$ can be equal to either of the following:
\begin{itemize}
\item[Case(i)] $(x_1^k, y_1^k, y_1^k, x_1^k, y_1^k, x_1^k, y_1^k)$,
\item[Case(ii)] $(x_1^k, y_1^k, x_1^k, y_1^k, x_1^k, x_1^k, y_1^k)$,
\item[Case(iii)] $(x_1^k, y_1^k, y_1^k, x_1^k, y_1^k, y_1^k, x_1^k)$,
\item[Case(iv)] $(x_1^k, y_1^k, x_1^k, y_1^k, x_1^k, y_1^k, x_1^k)$.
\end{itemize}
In either of these cases, we get either $a_{ij} \neq 0$ or $b_{ij} \neq 0$ such that $a_{ii}^k = a_{jj}^k$ or $b_{ii}^k = b_{jj}^k$ respectively for some $1\leq i,j \leq 7$. If we assume that for any $\alpha \in \mathbb{F}_q$ the equation $X_1^k + X_2^k = \alpha$ has two solutions in $\mathbb{F}_q^2$, say $(x_1, y_1)$, $(x_2, y_2)$ such that $x_1^k \neq x_2^k$ and $y_1^k \neq y_2^k$, then in $T_n(\mathbb{F}_q)$, $n \geq 7$, we start getting computational difficulties to write a matrix as the sum of two $k$-th powers. This also explains why the proof of our main Theorem given in the next section depends on both $k$ and $n$.

%%%%%%%%%%%
\section{Proof of main theorem}

Now we will prove the main result, Theorem~\ref{thma}, of the paper here.

\begin{proof}[\textbf{Proof of Theorem~\ref{thma}:}] Let $C =(c_{ij}) \in T_n(\mathbb{F}_q)$ be a matrix. We may assume $n>1$ and $k>1$. 

{\bf Proof of part 2: } Let us begin by showing the second part of the Theorem. 

First, let us assume that $C$ has only one eigenvalue, say $\lambda \in \mathbb F_q$. Then by Remark~\ref{solnn}, for all $q > n^2k^{16}$,  the equation $X_1^k+ X_2^k=\lambda$, has at least $n$ solutions  $(x_i,y_i) \in \mathbb{F}_q$, $1\leq i \leq n$, such that $x_i^k \neq x_j^k$ and $y_i^k \neq y_j^k$ for any $i\neq j$. Let $A$ and $B$ be  matrices in $T_n(\mathbb F_q)$ such that $a_{ii} = x_i^k$ and $b_{ii} = y_i^k$ respectively,  $a_{ij} = c_{ij} $ and $b_{ij}=0$, for all $i \neq j$. Then, by Lemma~\ref{mult}, matrices $A$ and $B$ are $k$-th powers. Hence $C = A + B$, where $A$ is a diagonalizable matrix and $B$ is a diagonal matrix and we are done. 

Now, if $C$ has all eigenvalues distinct, say $\lambda_i \in \mathbb F_q$, $1 \leq i \leq n$, then by Corollary~\ref{dia}, $C$ is conjugate to a diagonal matrix. Hence, we may assume $C$ is a diagonal matrix. By Remark~\ref{zero2}, for all $q > n^2k^{16}$ the equations $X_i^k + Y_i^k = \lambda_i$ have at least $n$ solutions, say $(x_i, y_i)$. If we take diagonal matrix $A$ and $B$ such that $a_{ii} = x_i^k$ and $b_{ii} = y_i^k$, respectively, then clearly $A$ and $B$ are $k$-th powers, also $C = A + B$. Note that, if $C$ is a diagonal matrix, then only one solution of the equation $X_i^k + Y_i^k=\lambda_i$ will be enough for us. 

Finally, let us assume that $C$ has $m <n$ distinct eigenvalues, say $\lambda_i \in \mathbb F_q$, such that the characteristic polynomial of $C$ is given by 
$$(X-\lambda_1)^{l_1}\cdots (X-\lambda_m)^{l_m},$$
where $l_1 + l_2 + \ldots + l_m =n$. First we will get $n$ ways to write $c_{ii} = x_{ii}^k + y_{ii}^k$ for some $x_{ii}, y_{ii} \in \mathbb F_q$ such that $x_{ii}^k \neq x_{jj}^k$ and $y_{ii}^k \neq y_{jj}^k$ for all $1 \leq i \neq j \leq n$.
Consider the system of equations over $\mathbb F_q$, 
\begin{eqnarray*}
X_1^k + Y_1^k &=& \lambda_1,\\ &\vdots&  \\
X_m^k+Y_m^k&=&\lambda_m,
\end{eqnarray*}
where $\lambda_i$ are distinct elements in $\mathbb{F}_q$. By Remark~\ref{solnn}, for every $\lambda_i$, we have the following: for all $q > \mathcal{C}(k,n)> 4n^2k^{16}$, there exist at least $2n$ many solutions, say $(x_i, y_i)$, of  the equation $X_i^k +Y_i^k= \lambda_i$ such that $x_i^k \neq x_j^k$ and $y_i^k \neq y_j^k$, for $i \neq j$. Let the number of such solutions for each $\lambda_i$ be $t_i$, where $t_i \geq 2n$. Thus, we get at least  $t_i$, many disjoint $V_{ir}$'s as we defined in Lemma \ref{2var1} and Lemma \ref{2var2}. By choosing one element from each $V_{ir}$, we get a set  
$$\mathcal A_i=\{(a_{i1},b_{i1}), (a_{i2}, b_{i2}), \ldots, (a_{i,t_i}, b_{i,t_i}) \mid a_{ir}^k \neq a_{is}^k,  b_{ir}^k \neq b_{is}^k \ \text{ for any }\ r \neq s\},
$$ 
where $t_i \geq 2n$. This is a subset of $S_i$, the set of solutions of equation  $X_i^k +X_i^k= \lambda_i$. Now choose first $l_1$ elements from $\mathcal A_1$, and set $ \mathcal{M}_1 := \{(a_{11},b_{11} ), \ldots, (a_{1, l_1}, b_{1,l_1}) \}$. As $|\mathcal A_2|=t_2 \geq 2n$ and $t_2-2l_1 \geq 2n - 2l_1>l_2$, therefore  by Lemma \ref{2var2}, we can choose $l_2$ many elements from $\mathcal A_2$, say  $(a_{21}, b_{21}), \ldots, (a_{2,l_2}, b_{2,l_2})$, by rearranging the elements of $\mathcal A_2$. Now define a set 
$$\mathcal{M}_2 = \{(a_{11},b_{11} ), \ldots, (a_{1, l_1}, b_{1,l_1}) , (a_{21}, b_{21}), \ldots, (a_{2,l_2}, b_{2,l_2}) \}.$$
It is clear from Lemma~\ref{2var2} and by our choice of elements that $a_{1r}^k \neq a_{2s}^k$ and
$b_{1r}^k \neq b_{2s}^k$ for any $1 \leq r \leq l_1$ and $1 \leq s \leq l_2$. Since $|\mathcal A_i| =t_i$ and $ t_i - 2(l_1 + \ldots + l_{i-1}) \geq 2n - 2(l_1+\ldots+l_{i-1})>l_i$, for all $ 2 \leq i\leq m$, therefore by Lemma \ref{2var2}, we can choose $l_i$ many elements to  construct $\mathcal{M}_i$ for all $i\geq 2$. Hence continuing like this, we get $\mathcal{M}_m$, which is given by 
$$\{ (a_{11},b_{11} ), \ldots, (a_{1, l_1}, b_{1,l_1}) , (a_{21}, b_{21}), \ldots, (a_{2,l_2}, b_{2,l_2}), \ldots,(a_{m1}, b_{m1}), \ldots, (a_{m,l_m}, b_{m,l_m}) \},
$$
such that $a_{ir}^k \neq a_{js}^k$ and $b_{ir}^k \neq b_{js}^k$ for all $i \neq j$ and $ r \neq s$. It is clear that $a_{ir}^k + b_{ir}^k =\lambda_i$, for all $1 \leq i \leq m$ and $1 \leq r \leq t_i$. 
Thus we  get $n$ many ways to write $c_{ii} = x_{ii}^k + y_{ii}^k$ for some $x_{ii}, y_{ii} \in \mathbb F_q$ such that $x_{ii}^k \neq x_{jj}^k$ and $y_{ii}^k \neq y_{jj}^k$ for all $1 \leq i \neq j \leq n$, where $(x_{ii}, y_{ii})= (a_{jr}, b_{jr})$ for some $1 \leq j \leq m$ and $1 \leq r \leq t_j$. Now let $A $ and $B$ be matrices in $T_n(\mathbb F_q)$ such that $a_{ii} =x_{jr}^k$, $a_{ij}=c_{ij}$ and $b_{ii}=y_{jr}^k$, $b_{ij}=0$, $x_{jr}^k +y_{jr}^k=\lambda_j$. It is clear that $C= A + B$. As diagonal entries of $A$ and $B$ are distinct $k$-th powers, therefore by Lemma \ref{mult} $A$ and $B$ are $k$-th powers.
Hence $C$ is sum of two $k$-th powers, whenever $q > 4n^2k^{16}$. 

By looking at the various constants in all of the cases above, we note that $C$ can be written as a sum of two $k$-th powers whenever $q > 4n^2k^{16}$. 

{\bf Proof of part 1: }
Now we will prove the first part of the theorem. Let $C$ be any matrix in $T_n(\mathbb F_q)$. If all diagonals enteries of $C$ are same, say $\lambda$, then by Corollary \ref{nvar1}, we get that for all $q> n^2k^{16}$, there exist at least $n $ many tuples $(x_i, y_i, z_i)$ such that $x_i^k +y_i^k +z_i^k = \lambda$ and $x_i^k \neq x_j^k$, $y_i^k \neq y_j^k$ for all $i \neq j$. Now consider  the matrix $A$  having $a_{ii} = x_i^k$, $a_{ij} = c_{ij}$, and $B, D$ are diagonal matrices with $b_{ii}=y_i^k$ and $d_{ii}=z_i^k$ respectively. Then, by Lemma \ref{mult}, we see that $A$ is a $k$-th power. Thus $C = A + B + D$, where $A,B$ and $D$ are $k$-th powers. 

If all diagonal entries of $C$ are distinct, then by Corollary \ref{nvar2}, we get $n$ many tuples $(x_i, y_i, z_i)$ such that $x_i^k +y_i^k +z_i^k = \lambda_i$ and $x_i^k \neq x_j^k$, $y_i^k \neq y_j^k$ for all $i \neq j$. Again by the same process, we can choose $A, B$ and $D$ such that $C = A+B+D$, where $A, B$ and $D$ are $k$-th powers. 

Now let us assume that, $C$ has $m <n$ distinct eigenvalues, say $\lambda_i \in \mathbb F_q$, such that the characteristic polynomial of $C$ is given by 
$$(X-\lambda_1)^{l_1}\cdots (X-\lambda_m)^{l_m},
$$
where $l_1+l_2+ \ldots + l_m =n$. First we will get $n$ many ways to write $c_{ii} = x_{ii}^k + y_{ii}^k$ for some $x_{ii}, y_{ii} \in \mathbb F_q$ such that $x_{ii}^k \neq x_{jj}^k$ and $y_{ii}^k \neq y_{jj}^k$ for all $1 \leq i \neq j \leq n$.
This can be done using Corollary \ref{nvar2} and a similar process, which we had done in the proof of the second part above. Thus for all $q>4n^2k^{16}$, we can write $C$ as a sum of three $k$-th powers. Hence we are done.

\end{proof}

We end this section with a brief discussion as to why some of the methods used for $M_n(\mathbb F_q)$ may not work here. 
Let $n\geq 1$ be a positive integer and $\lambda \in \mathbb{F}$. Let $J_{\lambda, n}$ denote the Jordan block matrix given as follows:
\[
\begin{bmatrix}
\lambda & 1 & 0 & \cdots & 0 & 0 \\
0 & \lambda & 1 & \cdots & 0 & 0 \\
0 & 0 & \lambda & \cdots & 0 & 0 \\
\vdots  & \vdots &\vdots & \ddots & \vdots & \vdots  \\
0 & 0 & 0 & \cdots & \lambda & 1 \\
0 & 0 & 0& \cdots & 0 & \lambda 
\end{bmatrix}.
\]

Let us begin with a discussion about powers and conjugation in $M_n(\mathbb{F}_q)$, see \cite[Lemma 5.1, 5.3]{KK22} of nilpotent matrices and can the same be carried forward in $T_n(\mathbb F_q)$. 
\begin{lemma}\label{J1}
Let $n,k \geq 1$ be positive integers, and suppose $n>k$. Let $m$ be $n$ congruent modulo $k$ with $0 \leq m \leq k-1$. Then the $k$-th   power $J_{0,n}^k$ is conjugate to $$\big( \bigoplus_{k-m} J_{0, \lfloor \frac{n}{k} \rfloor } \big) \bigoplus  \big( \bigoplus_{m} J_{0, \lceil \frac{n}{k} \rceil } \big).  $$ 
\end{lemma}

\begin{defn}
For a positive integer $n \geq 1$, the $r$-tuple $(n_1, \ldots, n_r)$, for $r \geq 1$, is a partition of $n$ if $1 \leq n_1 \leq \ldots \leq n_r$ and $n=n_1 + \ldots + n_r$. The coordinates $n_i$ are called parts of the partition. The Junction matrix associated to the partition $(n_1, \ldots, n_r)$ is the matrix
$$\mathfrak{J}_{(n_1, \ldots, n_r)}= E_{n_1, n_1+1} + E_{(n_1 +n_2),(n_1 +n_2+1)} + \cdots + E_{(n1 +\ldots + n_{r-1}),(n_1 + \ldots + n_{r-1}+1)}.$$ 
\end{defn} 

For example, corresponding to the partition $(1,1,1,2)$ and $(1,2,2)$ of $5$, the junction matrices are,

\[
\begin{bmatrix}
0 & 1 & 0 & 0 & 0  \\
0 & 0 & 1 & 0 & 0  \\
0 & 0 & 0 & 1 & 0  \\
0 & 0 & 0 & 0 & 0  \\
0 & 0 & 0 & 0 & 0  
\end{bmatrix}, 
\begin{bmatrix}
0 & 1 & 0 & 0 & 0  \\
0 & 0 & 0 & 0 & 0  \\
0 & 0 & 0 & 1 & 0  \\
0 & 0 & 0 & 0 & 0  \\
0 & 0 & 0 & 0 & 0  
\end{bmatrix}. 
\]

\begin{lemma}\label{J2}
For a positive integer $k \geq 1$, let $(n_1, \ldots, n_k)$ be a partition of $n$. Suppose $n \geq 2k$ and $n_i \geq 2$ for all $1 \leq i \leq r$. Then $\mathfrak{J}_{(n_1, \ldots, n_k)}$ is a $k$-th   power.
\end{lemma}

The Lemma \ref{J1} and Lemma \ref{J2}  need not hold true in $T_n(\mathbb{F}_q)$.  
If Lemma \ref{J1} is true in $T_n(\mathbb{F}_q)$, then $J_{(0,4)}^2 $ is conjugate to $J_{(0,2)} \oplus J_{(0,2)}$. But it can be easily checked that  $J_{(0,4)}^2 $ is not $B_n$-conjugate to $J_{(0,2)} \oplus J_{(0,2)}$. Also if Lemma \ref{J2} holds true in $T_n(\mathbb{F}_q)$, then for partition $(2,2)$ of $4$, we have $\mathfrak{J}_{(2,2)}$ is square of a matrix in $T_4(\mathbb{F}_q)$. But $\mathfrak{J}_{(0,2)}$ is not a square in $T_4(\mathbb{F}_q)$. 

%%%%%%%%%%%%
%\newpage
%%%%%%%%%%%%%%%%%%
\section{Tables}\label{section-table}
In this section, we will write indecomposable Bellitski\v{i}'s canonical forms of nilpotent matrices up to $6 \times 6$ as a sum of two $k$-th   powers. In the following tables, $E_{rs}$ denote the elementary matrix such that the $(r,s)$ entry is $1$ and elsewhere $0$. Here $D=\diag(d_{ii})$ and $D'=\diag (d_{jj}')$ are diagonal matrices with diagonal entries $x_i^k$ and $y_i^k$ respectively such that $x_i^k + y_i^k = 0$, where $x_i, y_i \in \mathbb{F}_q$, $ i \in \{1, 2\}$ and $k$ be any positive integer. In the following table, $A=D+\sum E_{ij}$ and $B=D'+\sum E_{rs}$, where $A$ and $B$ are $k$-th  powers by Lemma \ref{power}.

\FloatBarrier
\begin{table}[ht]
\centering
\begin{tabular}{|c|c|c|}
\hline
Presentation &  \multicolumn{1}{|m{6cm}|}{\centering $A+B$ }  & \multicolumn{1}{|m{5cm}|}{\centering Diagonal } \\ 
\hline
&&\\
$123$ & $(D+E_{12})+(D'+E_{23})$ &  $d_{11}=d_{33}=x_1^k$, $d_{22}=x_2^k$   \\
\hline
&&\\
$1234$ & $(D+E_{12}+E_{34})+$ &  $d_{11}=d_{33}=x_1^k$,   \\
&&\\
 & $(D'+E_{23})$ &   $d_{22}=d_{44}=x_2^k$   \\
\hline
&&\\
$12|34:13$ & $(D+E_{12}+E_{34})$+ &  $d_{11}=d_{44}=x_1^k$,   \\
&&\\
 & $(D'+E_{13})$ &   $d_{22}=d_{33}=x_2^k$   \\
\hline
&&\\
$12345$ &   $(D+ E_{12}+ E_{34})$ +   &  $d_{11}=d_{33}=d_{55}=x_1^k$,   \\
&&\\
 &   $(D'+E_{23} + E_{45})$   &   $d_{22}=d_{44}=x_2^k$   \\
\hline
&&\\
$12|345:13$ & $(D+E_{12}+E_{13}+E_{45})$+  &    $d_{11}=d_{44}=x_1^k$  \\
&&\\
 &   $(D'+E_{34})$  &    $d_{22}=d_{33}=d_{55}= x_2^k$  \\
\hline
&&\\
$123|45:24$ &   $(D+E_{12}+E_{45})$+  &  $d_{11}=d_{33}=d_{44}=x_1^k$,   \\
&&\\
 &  $(D'+E_{23}+E_{24})$  &  $d_{22}=d_{55}=x_2^k$   \\
\hline
&&\\
$145|23:24$ &  $(D+E_{14}+E_{23}+E_{24})$ +  &  $d_{11}=d_{22}=d_{55}=x_1^k$, \\
&&\\
 &  $(D'+E_{45})$  &  $d_{44}=d_{33}=x_2^k$   \\
\hline
&&\\
$125|34:13$ &  $(D+E_{12}+E_{13})$+ &  $d_{11}=d_{44}=d_{55}=x_1^k$,   \\
&&\\
 &  $(D'+E_{25}+E_{34})$  &  $d_{22}=d_{33}=x_2^k$   \\
\hline
\end{tabular}
\vspace{.25cm}
\caption{$T_n(\mathbb{F}_q), n \leq 5$.}
\label{table1}
\end{table}

\begin{table}[ht]
\centering
\renewcommand{\arraystretch}{.9}
\begin{tabular}{|c|c|c|}
\hline
Presentation &  \multicolumn{1}{|m{6cm}|}{\centering $A^k+B^k$ }  & \multicolumn{1}{|m{5cm}|}{\centering Diagonal } \\ 
\hline
&&\\
$123456$ &  $(D+E_{12}+E_{34}+E_{56})$+ &   $d_{11}=d_{33}=d_{55}=x_1^k$,  \\
&&\\
 &  $(D'+E_{23}+E_{45})$  &    $d_{22}=d_{44}=d_{66}=x_2^k$    \\
\hline
&&\\
$12|3456:13$ &  $(D+E_{12}+E_{13}+E_{45})$+  &  $d_{11}=d_{44}=d_{66}=x_1^k$,  \\
&&\\
 &  $(D'+E_{34}+E_{56})$  &   $d_{22}=d_{33}=d_{55}=x_2^k$    \\
\hline
&&\\
$123|456:14$ &  $(D+E_{12}+E_{14}+E_{56})$+  &  $d_{11}=d_{33}=d_{55}=x_1^k$,  \\
&&\\
 &  $(D'+E_{23}+E_{45})$  &   $d_{22}=d_{44}=d_{66}=x_2^k$    \\
\hline
&&\\
$1456|23:24$ &  $(D+E_{14}+E_{24}+E_{56})$+  &  $d_{11}=d_{22}=d_{55}=x_1^k$,   \\
&&\\
 &  $(D'+E_{23}+E_{45})$  &   $d_{33}=d_{44}=d_{66}=x_2^k$    \\
\hline
&&\\
$123|456:24$ &  $(D+E_{23}+E_{24}+E_{56})$+  &  $d_{11}=d_{33}=d_{44}=d_{66}=x_1^k$,   \\
&&\\
 &  $(D'+E_{12}+E_{45})$  &   $d_{22}=d_{55}=x_2^k$    \\
\hline
&&\\
$124|356:13$ &  $(D+E_{12}+E_{13}+E_{56})$+  &  $d_{11}=d_{44}=d_{55}=x_1^k$, \\
&&\\
 &  $(D'+E_{24}+E_{35})$  &   $d_{22}=d_{33}=d_{66}=x_2^k$    \\
\hline
&&\\
$14|23|56:15|25$ &  $(D+E_{14}+E_{23}+E_{56})$+  &  $d_{11}=d_{22}=d_{66}=x_1^k$,   \\
&&\\
 &  $(D'+E_{15}+E_{25})$  &   $d_{33}=d_{44}=d_{66}=x_2^k$   \\
\hline
&&\\
$1256|34:13$ &  $(D+E_{12}+E_{13}+E_{56})$+ &  $d_{11}=d_{44}=d_{55}=x_1^k$, \\
&&\\
 &  $(D'+E_{25}+E_{24})$  &   $d_{22}=d_{33}=d_{66}=x_2^k$   \\
\hline
&&\\
$12|34|56:13|35$ &  $(D+E_{12}+E_{13}+E_{56})$+  &  $d_{11}=d_{44}=d_{55}=x_1^k$,    \\
&&\\
 &  $(D'+E_{34}+E_{35})$  &   $d_{22}=d_{33}=d_{66}=x_2^k$    \\
\hline
\end{tabular}
\vspace{.25cm}
\caption{$T_6(\mathbb{F}_q)$}
\label{table2}

\end{table}
\FloatBarrier

\begin{table}[ht]
\centering
\renewcommand{\arraystretch}{.9}
\begin{tabular}{|c|c|c|}
\hline
Presentation &  \multicolumn{1}{|m{6cm}|}{\centering $A^k+B^k$ }  & \multicolumn{1}{|m{5cm}|}{\centering Diagonal } \\ 
\hline
&&\\
$134|256:35$ &  $(D+E_{25}+E_{34}+E_{35})$+  &  $d_{11}=d_{44}=d_{55}=x_1^k$    \\
&&\\
 &  $(D'+E_{13}+E_{56})$  &   $d_{22}=d_{33}=d_{66}=x_2^k$    \\
\hline
&&\\
$156|234:35$ &  $(D+E_{15}+E_{34}+E_{35})$+  &  $d_{11}=d_{33}=d_{66}=x_1^k$,     \\
&&\\
 &  $(D'+E_{23}+E_{56})$  &  $d_{22}=d_{44}=d_{55}=x_2^k$     \\
\hline
&&\\
$1234|56:34$ &  $(D+E_{12}+E_{34}+E_{35})$+  &  $d_{11}=d_{33}=d_{66}=x_1^k$,   \\
&&\\
 &  $(D'+E_{23}+E_{56})$  &  $d_{22}=d_{44}=d_{55}=x_2^k$    \\
\hline
&&\\
$12|36|45:13|14$ &  $(D+E_{12}+E_{13}+E_{14})$+  &  $d_{11}=d_{55}=d_{66}=x_1^k$,    \\
&&\\
 &  $(D'+E_{45}+E_{36})$  &  $d_{22}=d_{33}=d_{44}=x_2^k$    \\
\hline
&&\\
$145|236:24$ &  $(D+E_{14}+E_{23}+E_{24})$+  &  $d_{11}=d_{22}=d_{55}=d_{66}=x_1^k$,  \\
&&\\
 &  $(D'+E_{45}+E_{36})$  & $d_{22}=d_{33}=d_{66}=x_2^k$   \\
\hline
&&\\
$1236|245$ &  $(D+E_{12}+E_{36}+E_{45})$+  &  $d_{11}=d_{33}=d_{44}=x_1^k$,  \\
&&\\
&  $(D'+E_{23}+E_{34})$  &  $d_{22}=d_{55}=d_{66}=x_2^k$   \\
\hline
&&\\
$126|345:13$ &   $(D+E_{12}+E_{13}+E_{45})$+  &   $d_{11}=d_{44}=d_{66}=x_1^k$, \\
&&\\
 &   $(D'+E_{26}+E_{34})$  &  $d_{22}=d_{33}=d_{55}=x_2^k$    \\
\hline
&&\\
$1256|34:13$ &   $(D+E_{12}+E_{13}+E_{45})$+  &   $d_{11}=d_{44}=d_{55}=x_1^k$,\\ 
&&\\
&   $(D'+E_{25}+E_{34})$  &  $d_{22}=d_{33}=d_{66}=x_2^k$    \\
\hline
&&\\
$1256|34:13|35$ &  $(D+E_{12}+E_{13}+E_{56})$+  &  
 $d_{11}=d_{44}=d_{55}=x_1^k$,\\ 
&&\\
 &  $(D'+E_{25}+E_{34}+E_{35})$  &  
$d_{22}=d_{33}=d_{66}=x_2^k$    \\
\hline
\end{tabular}
\vspace{.25cm}
\caption{$T_6(\mathbb{F}_q)$}
\label{table3}

\end{table}
\FloatBarrier


\begin{thebibliography}{9}



\bibitem{BMRY20} A. K. Belov, S. Malev, L. Rowen and R. Yavich,  \emph{Evaluations of noncommutative polynomials on algebras: methods and problems, and the {L}'vov-{K}aplansky conjecture}, SIGMA. Symmetry, Integrability and Geometry. Methods and Applications, \textbf{16}, (2020) 61.

\bibitem{BS22} M. Bre\v{s}ar  and P. Semrl, \emph{The Waring Problem for Matrix Algebras}, Israel J. Math. \textbf{253} (2023), no.1, 381--405.

\bibitem{BS23} M. Bre\v{s}ar  and P. Semrl, \emph{The Waring Problem for Matrix Algebras- II}, Bull. Lond. Math. Soc. \textbf{55} (2023), no.4, 1880--1889. 


\bibitem{CX16}
Y. Chen, Y. Xu, H. Li, W. Fu, \emph{Belitski\v{i}’s canonical forms of upper triangular nilpotent matrices under upper triangular similarity}, Linear Algebra and its Applications,  \textbf{506},  (2016), 139--153.


\bibitem{DSY21} S. Dolfi, A. Singh, M. K. Yadav,  \emph{ p-power conjugacy classes in U(n,q) and T(n,q)}, J. Algebra Appl. no.7, \textbf{20} (2021) Paper No. 2150121, 13 pp.


\bibitem{Ga21} A. S. Garge,  Matrices over commutative rings as sums of fifth and seventh powers of matrices, Linear Multilinear Algebra, \textbf{69}, (2021) 12, 2220--2227. 

\bibitem{KG13} S. A. Katre  and A. S. Garge, \emph{Matrices over commutative rings as sums of {$k$}-th powers}, Proc. Amer. Math. Soc., \textbf{141}, 2013, 1, 103--113. 


\bibitem{KK22}
 K. Kishore, \emph{ Matrix Waring Problem,} Linear Algebra and its Applications,  \textbf{646},  (2022), 84--94.

 \bibitem{KA22}
 K. Kishore, A. Singh, \emph{ Matrix Waring Problem-II,} to appear, Israel J. Math. (2023).

\bibitem{KO05}
D. Kobal, \emph{Belitski\v{i}’s canonical form for 5 × 5 upper
triangular matrices under upper triangular similarity}, Linear Algebra and its Applications \textbf{403} (2005) 178--182


 
\bibitem{LST19} M. Larsen, A. Shalev and P. H. Tiep,  \emph{Probabilistic {W}aring problems for finite simple groups}, Ann. of Math. (2), \textbf{190}, (2019) 2, 561--608. 

\bibitem{LST11} M. Larsen, A. Shalev and P. H. Tiep, \emph{The {W}aring problem for finite simple groups}, Ann. of Math. (2), \textbf{174}, (2011), 3, 1885--1950. 

\bibitem{LW54} Lang, Serge; Weil, Andre, \emph{Number of points of varieties in finite fields}, Amer. J. Math. \textbf{76} (1954), 819--827.

\bibitem{SP23} S. Panja and S. Prasad, \emph{The image of polynomials and Waring type problems on upper triangular matrix algebras}, Journal of Algebra, \textbf{631}, 2023, 148--193

\bibitem{Ri87} D. R. Richman,  \emph{The {W}aring problem for matrices}, Linear and Multilinear Algebra, \textbf{22}, 1987, 2, 171--192. 

 \bibitem{sch}
 W.~M.~ Schmidt,  \emph{ Equations over finite fields an elementary approach,} Lecture notes in mathematics, Springer-Verlag, Berlin-Heidberg-New York (1976). 

\bibitem{So16}
Y. V. Sosnovskiy, \emph{On the width of verbal subgroups of the groups of triangular matrices over a field of arbitrary characteristic}, Int. J. Algebra Comput. \textbf{26} (2), 2016, 217--222.  


\bibitem{MMH} M. C. Tsai, M. Bogale and H. Huang, \emph{ On Triangular Similarity of Nilpotent Triangular Matrices}, Linear Algebra and its Applications, \textbf{596}, 2020, Pages 1--35.

   
\bibitem{Va87} L. N. Vaserstein,  \emph{On the sum of powers of matrices}, Linear and Multilinear Algebra, \textbf{21}, (1987), 3, 261--270. 

\bibitem{VA91} Vera-López, A.; Arregi, J. M., \emph{Conjugacy classes in Sylow p-subgroups of GL(n,q). II}, Proc. Roy. Soc. Edinburgh Sect. A \textbf{119} (1991), no. 3-4, 343--346.

\bibitem{VA92} Vera López, Antonio; Arregi, Jesus Maria, \emph{Conjugacy classes in Sylow p-subgroups
of GL(n, q)}, J. Algebra \textbf{152} (1992), no. 1, 1--19.

\bibitem{VAGO17} Vera-López, A.; Arregi, J. M.; García-Sánchez, M. A.; Ormaetxea, L., \emph{On Higman's conjecture}, Electron. J. Linear Algebra \textbf{32} (2017), 151--162.

\end{thebibliography}
\end{document}